\documentclass[oneside,11pt]{amsart}
\usepackage[latin9]{inputenc}
\usepackage{amstext}
\usepackage{amsthm}
\usepackage{stmaryrd, theoremref}
\usepackage[margin=1in]{geometry}
\usepackage[citestyle=alphabetic,bibstyle=alphabetic,doi=false,isbn=false,url=false]{biblatex}
\makeatletter
\numberwithin{equation}{section}
\numberwithin{figure}{section}
\theoremstyle{plain}
\newtheorem{thm}{\protect\theoremname}
\theoremstyle{plain}
\newtheorem{cor}[thm]{\protect\corollaryname}
\theoremstyle{definition}
\newtheorem{defn}[thm]{\protect\definitionname}
\theoremstyle{plain}
\newtheorem{prop}[thm]{\protect\propositionname}
\theoremstyle{plain}
\newtheorem{lem}[thm]{\protect\lemmaname}
\theoremstyle{remark}
\newtheorem{rem}[thm]{\protect\remarkname}

\numberwithin{thm}{section}

\@ifundefined{date}{}{\date{}}
\makeatother

\providecommand{\corollaryname}{Corollary}
\providecommand{\definitionname}{Definition}
\providecommand{\lemmaname}{Lemma}
\providecommand{\propositionname}{Proposition}
\providecommand{\remarkname}{Remark}
\providecommand{\theoremname}{Theorem}

\begin{document}
\title{The Entropy of Ricci Flows with Type-I Scalar Curvature Bounds}
\author{Max Hallgren}

\maketitle

\begin{abstract}
In this paper, we extend the theory of Ricci flows satisfying a Type-I scalar curvature bound at a finite-time singularity. In \cite{bam2}, Bamler showed that a Type-I rescaling procedure will produce a singular shrinking gradient Ricci soliton with singularities of codimension 4. We prove
that the entropy of a conjugate heat kernel based at the singular time converges to the soliton entropy of the singular soliton, and use this to characterize the singular set of the Ricci flow solution in terms of a heat kernel density function.  This generalizes results previously only known with the stronger assumption of a Type-I curvature bound. We also show that in dimension 4, the singular Ricci soliton is smooth away from finitely many points, which are conical smooth orbifold singularities.
\end{abstract}

\section{Introduction}

This paper is concerned with the finite-time singularities of solutions $(M^{n},(g_{t})_{t\in[0,T)})$
to the Ricci flow
\[
\dfrac{\partial}{\partial t}g_{t}=-2Rc(g_{t})
\]
on a closed manifold which satisfy the Type-I scalar curvature condition
\begin{equation} \label{eq:scal}
\limsup_{t\to T}\max_{M}R(\cdot,t)(T-t)<\infty,
\end{equation}
where $T<\infty$ is the maximal existence time.  In particular, we generalize some of the theory of Ricci flow solutions
which satisfy the more stringent Type-I curvature assumption 
\begin{equation} \label{eq:curv}
\limsup_{t\to T}\max_{M}|Rm|(\cdot,t)(T-t)<\infty.
\end{equation}
Ricci flow solutions satisfying \eqref{eq:curv} have been studied
in \cite{topping,mant,naber}, where it was shown that,
for any fixed $q\in M$ and sequence of times $t_{i}\nearrow T$,
a subsequence of $(M^{n},(T-t_{i})^{-1}g_{t_{i}},q)$ converges in
the pointed Cheeger-Gromov sense to a complete Riemannian manifold
$(M_{\infty},g_{\infty},q_{\infty})$ equipped with a function $f_{\infty}\in C^{\infty}(M_{\infty})$
which satisfies the shrinking gradient Ricci soliton (GRS) equation
\[
Rc_{g_{\infty}}+\nabla^{2}f_{\infty}=\frac{1}{2}g_{\infty}.
\]
While it is unknown whether this limiting soliton is uniquely determined
by the basepoint $q$, in \cite{mant} it is shown that all such solitons
share a numerical invariant, called the shrinker entropy $\mathcal{W}(g_{\infty},f_{\infty})$,
(see Section 4), which is determined by $q$. They also show that
$\mathcal{W}(g_{\infty},f_{\infty})=0$ if and only if $(M_{\infty},g_{\infty},f_{\infty})$
is the Gaussian shrinking soliton on flat Euclidean space.

While interesting questions about solutions satisfying \eqref{eq:curv} condition remain open, this condition is often too restrictive
for useful applications of Ricci flow to geometry and topology. Condition \eqref{eq:scal}, on the other hand, is satisfied
by K\"ahler-Ricci flow on a Fano manifold with initial metric in the
canonical K\"ahler class by the work of Perelman (see \cite{kahler}), and is conjectured
to be satisfied for a much larger class of K\"ahler-Ricci flow solutions
(Conjecture 7.7 of \cite{songnotes}). 
One of the main technical difficulties that arises when studying Ricci flows satisfying \eqref{eq:scal} is that one cannot expect Type-I blowups to result in a smooth limiting space. In fact, most results about Ricci flows satisfying \eqref{eq:curv}, including \cite{cao,topping,naber,mant}, depend crucially on applying the Cheeger-Gromov compactness theorems to rescaled solutions. However, in \cite{bam1},\cite{bam2}, Bamler develops an extensive theory for taking weak limits of Ricci flows with uniformly bounded scalar curvature, modeled on the Cheeger-Colding-Naber-Tian theory of noncollapsed Riemannian manifolds with bounded Ricci curvature. In particular, Theorem 1.2 of \cite{bam2} shows that any Ricci flow satisfying \eqref{eq:scal} has a dilation limit which is a singular space (see section 2), and which
possesses the structure of a smooth but incomplete shrinking Ricci soliton outside of a subset of Minkowski codimension 4. 

The main goal of this paper is to extend Bamler's analysis of the singular limits of dilated Ricci flows satisfying \eqref{eq:scal}, and to relate some of their properties to the original Ricci flow. The first main theorem generalizes the aforementioned results in \cite{mant}. 

In order to state this theorem, we first recall a result in \cite{bam2}. Assume $(M^n,(g_t)_{t\in [0,T)})$ is a closed, pointed solution of Ricci flow satisfying (\ref{eq:scal}), and fix any sequence $t_i \nearrow T$. According to Theorem 1.2 of \cite{bam2}, we can pass to a subsequence to get pointed Gromov-Hausdorff convergence of $(M^n,(T-t_i)^{-1}g_{t_i},q)$ to a pointed singular space $(\mathcal{X},q_{\infty})=(X,d,\mathcal{R},g_{\infty},q_{\infty})$. Moreover, there exists $f_{\infty} \in C^{\infty}(\mathcal{R})$ obtained as a limit of rescalings of a conjugate heat kernel based at the singular time, which satisfies the Ricci soliton equation on $\mathcal{R}$. The Ricci soliton $(\mathcal{R},g_{\infty},f_{\infty})$ has a well-defined entropy $\mathcal{W}(g_{\infty},f_{\infty})$, defined in Section 4, and there is a heat kernel density function (defined in Section 3) $\Theta(q)\in (-\infty,0]$ associated to the basepoint $q$.

\begin{thm} \label{mainthm} $\Theta(q)=\mathcal{W}(g_{\infty},f_{\infty})$, with $\Theta(q)=0$ if and only if $(\mathcal{R},g_{\infty},f_{\infty})$
is the Gaussian shrinker on flat $\mathbb{R}^{n}$, in which case
there is a neighborhood $U$ of $q$ in $M$ such that $\sup_{U\times[-2,0)}|Rm|<\infty$. 
\end{thm}

In particular, all singular shrinking GRS which arise as Type-I dilation
limits at a fixed point in $M$ possess the same shrinker entropy.
We recall the definition of the singular set of $(M,(g_{t})_{t\in[0,T)})$,
defined in \cite{topping} as
\[
\Sigma:=\left\{ x\in M;\sup_{U\times[0,T)}|Rm|=\infty\text{ for every neighborhood }U\text{ of }q\text{ in }M\right\} .
\]
In the general Riemannian case, little is known about the regularity
or structure of $\Sigma$. In the case where $(M,g_{0})$ is K\"ahler,
it is known that $\Sigma$ is actually an analytic subvariety of $M$, even without the Type-I assumption (see \cite{collinstosatti}). With
a Type-I curvature assumption, it was shown in \cite{mant} that $\Sigma$
is characterized by the density function: $\Sigma=\Theta^{-1}(0)$.
We are able to generalize this result to the case of Type-I scalar
curvature bounds.
\begin{cor} \label{rigiditythm}
Suppose $(M^{n},(g_{t})_{t\in[0,T)},q)$ is a closed, pointed solution
of Ricci flow satisfying \eqref{eq:scal}. Then
$\Sigma=\Theta^{-1}(0)$.
\end{cor}

Finally, in dimension 4, we extend Bamler's results on the structure
of singular shrinking GRS by giving a more precise description of
the singular part of the shrinking soliton. We let $(X,d,\mathcal{R},g_{\infty})$ be the singular space of Theorem \ref{mainthm} and the discussion preceding it.
\begin{thm} \label{orbtheorem}
If $n=4$, then $X\setminus\mathcal{R}$
consists of finitely many points, and $X$ has the structure
of a $C^{\infty}$ Riemannian orbifold.
\end{thm}

In particular, if $\mathcal{X}=(X,d,\mathcal{R},g)$ is the singular space in Theorem \ref{orbtheorem}, then there exists $f\in C^{\infty}(\mathcal{R})$ such that $(\mathcal{R},g,f)$ is an incomplete but smooth shrinking
GRS, and each $x\in X\setminus\mathcal{R}$ admits a finite group
$\Gamma_{x}\subseteq\mathbb{R}^{4}$ acting linearly and freely away from the origin, along with a homeomorphism
$\varphi_{x}:\mathbb{R}^{4}/\Gamma_{x}\supseteq B(0^4,r_0) \to B^{X}(x,r_0)$ such that,
if $\pi_{x}:\mathbb{R}^{4}\to\mathbb{R}^{4}/\Gamma_{x}$ is the quotient
map, then $\varphi_{x}\circ\pi_{x}$ is a smooth map on $B(0^4,r_0)\setminus\{0\}$,
and $(\varphi_{x}\circ\pi_{x})^{\ast}g$, $(\varphi _x \circ \pi_x )^{\ast}f$ extend smoothly to a Riemannian metric and function on $B(0^4,r_0)$.

Theorem \ref{orbtheorem} was proved in the setting of Fano K\"ahler-Ricci flow in \cite{bingspace1}, where
it was essential that the $L^{2}$ norm of the curvature tensor is
uniformly bounded along the flow. This fails in the general Riemannian
setting (even if we assume \eqref{eq:curv}), so our proof must rely on different arguments.

In Section 2, we collect definitions and results related to Ricci
flows satisfying certain scalar curvature bounds. In Section 3, we
establish Gaussian-type estimates for conjugate heat kernels based
at the singular time, largely along the lines of Bamler and Zhang's
heat kernel estimates. In Section 4, we define shrinker entropy, and
prove an important integration-by-parts lemma for singular shrinking
GRS. In Section 5, we prove the convergence of entropy and the heat
kernel measure. In Section 6, we show that the shrinker entropy of
a normalized singular GRS only depends on the underlying manifold,
and use this to complete the proof of Theorem \ref{mainthm} and Corollary \ref{rigiditythm}. Finally,
in Section 7, we specialize to the case of dimension 4, and prove Theorem \ref{orbtheorem}.

The author would like to thank his advisor Xiaodong Cao for his helpful feedback and support, as well as Jian Song for useful discussions.

\section{Preliminaries and Notation}

Given a solution $(M^{n},(g_{t})_{t\in[0,T)})$ of Ricci flow, we
let $d_{t}:M\times M\to[0,\infty)$ be the length metric induced by
$g_{t}$, and define
\[
B(x,t,r):=B_{g_{t}}(x,r):=\{y\in M;d_{t}(x,y)<r\},
\]
\[
Q^{+}(x,t,r):=\{(y,s)\in M\times[t,t+r^{2}];d_{s}(y,x)<r\},
\]
\[
r_{Rm}^g(x,t):=r_{Rm}(x,t):=\sup\{r>0;|Rm|\leq r^{-2}\text{ on }B(x,t,r)\}.
\]
for all $(x,t)\in M\times[0,T)$ and $r>0$. For measurable $S\subseteq M$,
we set $|S|_{t}:=\text{Vol}_{g_{t}}(S)$. We denote the Lebesgue measure
on a Riemannian manifold $(M,g)$ as $dg$. If we consider a rescaled flow, for example $\widetilde{g}_t=\lambda g_{\lambda^{-1}t}$, we let $\widetilde{d}_t$ be the length metric induced by $\widetilde{g}_t$, $\widetilde{B}(x,t,r):=B_{\widetilde{g}_t}(x,r)$ the corresponding geodesic ball, and so on. If $(X,d)$ is a metric space, we also set
$$B^X(x,r):=\{y\in X;d(x,y)<r\}.$$
If in addition $\text{diam}(X)\leq \pi$, then we denote by $(C(X),d_{C(X)},c_0)$ the corresponding metric cone, with vertex $c_0$.

We recall Perelman's $\mathcal{W}$ functional, defined by
\begin{align*}
\mathcal{W}(g,f,\tau)= & (4\pi\tau)^{-\frac{n}{2}}\int_{M}(\tau(R+|\nabla f|^{2})+f-n)e^{-f}dg
\end{align*}
for any Riemannian metric $g$ on $M$, and any $f\in C^{\infty}(M),$$\tau>0$.
For any compact Riemannian manifold, Perelman's invariants
\[
\mu[g,\tau]:=\inf\left\{ \mathcal{W}(g,f,\tau);f\in C^{\infty}(M)\text{ and }\int_{M}(4\pi\tau)^{-\frac{n}{2}}e^{-f}dg=1\right\} ,
\]

\[
\nu[g,\tau]:=\inf_{s\in[0,\tau]}\mu[g,s].
\]
Note that this definition of $\nu$ is not completely standard.

We now define the class of weak limit spaces we will be considering, following the definitions in \cite{bam1}, \cite{bam2}.  
\begin{defn}
\noindent A singular space is a tuple $\mathcal{X}=(X,d,\mathcal{R},g)$,
where $(X,d)$ is a complete, locally compact metric length space,
and $(\mathcal{R},g)$ is a $C^{\infty}$ Riemannian manifold satisfying
the following:\\
 $(i)$ $d|(\mathcal{R}\times\mathcal{R})$ is the length metric of
$(\mathcal{R},g)$.\\
 $(ii)$ $\mathcal{R}$ is an open, dense subset of $X$.\\
 $(iii)$ for any compact subset $K\subseteq X$ and $D\in(0,\infty)$,
there exists $\kappa=\kappa(K,D)>0$ such that, for all $x\in K$
and $r\in(0,D)$, we have 
\[
\kappa r^{n}\leq|B^X(x,r)\cap\mathcal{R}|\leq\kappa^{-1}r^{n}.
\]
$\mathcal{X}$ is said to have singularities of codimension $p_{0}>0$
if, for all $p\in(0,p_{0})$, $x\in X$ and $r_{0}>0$, there exists
$E_{p,x,r}<\infty$ such that 
\[
|\{r_{Rm}<rs\}\cap B^X(x,r)\cap\mathcal{R}|\leq E_{p,x,r}r^{n}s^{p}
\]
for all $r\in(0,r_{0})$, $s\in(0,1)$. $\mathcal{X}$ is said to
have mild singularities if, for any $p\in\mathcal{R}$, there exists
a closed subset $\mathcal{Q}_{p}\subseteq\mathcal{R}$ of measure
zero such that, for any $x\in\mathcal{Q}_{p}$, there exists a minimizing
geodesic from $p$ to $x$ lying entirely in $\mathcal{R}$. $\mathcal{X}$
is $Y$-regular at scale $a$ if, for any $x\in X$ and $r\in(0,a)$
satisfying $|B(x,r)\cap \mathcal{R}|>(\omega_{n}-Y^{-1})r^{n}$, we have $r_{Rm}(x)>Y^{-1}r$. 
\end{defn}

\begin{defn}
If $(M_{i},g_{i},q_{i})$ is a sequence of complete, pointed Riemannian
manifolds and $(\mathcal{X},q_{\infty})=(X,d,\mathcal{R},g,q_{\infty})$
is a pointed singular space, a convergence scheme $(U_{i},V_{i},\phi_{i})$
for the convergence $(M_{i},g_{i},q_{i})\to(\mathcal{X},q_{\infty})$
consists of open subsets $V_{i}\subseteq M_{i}$, $U_{i}\subseteq\mathcal{R}$,
and diffeomorphisms $\phi_{i}:U_{i}\to V_{i}$ such that the following
hold:\\
$(i)$ $(U_{i})$ is an increasing sequence with $\bigcup_{i}U_{i}=\mathcal{R}$,\\
$(ii)$ $\phi_{i}^{\ast}g_{i}\to g$ in $C_{loc}^{\infty}(\mathcal{R})$,\\
$(iii)$ there exist $q_{i}'\in U_{i}$ with $q_{i}'\to q_{\infty}$
and $d_{M_{i}}(q_{i},\phi_{i}(q_{i}'))\to0$,\\
$(iv)$ for any $D<\infty$ and $\epsilon>0$, there exists $i_{0}=i_{0}(D,\epsilon)\in\mathbb{N}$
such that for all $i\geq i_{0}$ and $x_{1},x_{2}\in B^{X}(q_{\infty},D)\cap U_{i}$,
we have 
\[
|d_{X}(x_{1},x_{2})-d_{M_{i}}(\phi_{i}(x_{1}),\phi_{i}(x_{2}))|<\epsilon,
\]
and such that, for any $y\in B^{M_{i}}(q_{i},D)$, there exists $x\in U_{i}$
such that $d_{M_{i}}(\phi_{i}(x),y)<\epsilon$.

Note that conditions $(iii)$,$(iv)$ imply that $\phi_i$ are Gromov-Hausdorff approximations. If  a convergence scheme exists, we say that $(M_{i},g_{i},q_{i})$
converges to $(\mathcal{X},q_{\infty})$. 
\end{defn}

We will commonly rely on the the main theorem of Bamler in \cite{bam2},
which establishes weak uniqueness properties and integral curvature
bounds for Ricci flows with bounded scalar curvature.

\begin{thm}
\thlabel{bamthm}
\noindent (Theorems 1.4, 1.7 in \cite{bam2}) Suppose $(M_{i}^{n},(g_{t}^{i})_{t\in[-2,0]},q_{i})$ is
a sequence of closed solutions of Ricci flow satisfying the following:\\
 i. $\nu[g_{-2}^{i},4]\geq-A$,\\
 ii. $|R_{g^i}|\leq\rho_{i}\leq A$ on $M_{i}\times[-2,0]$,\\
where $A<\infty$. Then some subsequence of $(M_{i},g_{0}^{i},x_{i})$ converges to a
pointed singular space $(\mathcal{X},q_{\infty})$ with singularities
of codimension 4, which is $Y$-regular at the scale 1, where $Y=Y(n,A)<\infty$.
Also, for any $\epsilon>0$, there exists $C=C(A,\epsilon,n)<\infty$
such that 
\[
\int_{B(x,t,r)}(r_{Rm}(\cdot,t))^{-4+2\epsilon}dg_{t}^{i}<Cr^{n-4+2\epsilon}.
\]
for any $r\in (0,1]$.
\noindent If in addition $\rho_{i}\to0$, then $\mathcal{X}$ is Ricci
flat and has mild singularities. 
\end{thm}

\noindent The following estimate is a consequence of Perelman's no-local-collapsing
theorem \cite{perl}, Qi Zhang's non-inflating theorem (Theorem 1.1 of \cite{zvol}),
and a basic covering argument (Lemma 2.1 of \cite{BZ2}).\\

\begin{prop} \label{volumebounds}
For any $A<\infty$, there exists $C=C(A)<\infty$ such that, for
any closed solution $(M^{n},(g_{t})_{t\in[-2,0)})$ of Ricci flow
satisfying:\\
 $(i)$ $\nu[g_{-2},4]\geq-A$,\\
 $(ii)$ $|R|\leq A$ on $M\times[-2,0)$ ,\\
then for any $(x,t)\in M\times[-1,0]$, $r>0$, we have 
\[
C^{-1}(\min\{1,r\})^{n}\leq|B(x,t,r)|_{t}\leq Cr^{n}e^{Cr}.
\]
If instead of $(ii)$ we have $|R|\leq A|t|^{-1}$ on $M$ for all
$t\in[-2,0)$, then 
\[
|B(x,t,r)|_{t}\leq Cr^{n}
\]
for all $r\in(0,1]$. 
\end{prop}

We will also need the following distortion estimate for Ricci flows with bounded scalar curvature.

\begin{thm}[Theorem 1.1 in \cite{BZ2}] \label{distortion} Given $A<\infty$, there exists $B=B(A,n)<\infty$ such that the following holds. Suppose $(M^n,(g_t)_{t\in [-2,0]})$ is a closed Ricci flow satisfying:\\
$(i)$ $\nu[g_{-2},4]\geq -A$,\\
$(ii)$ $|R| \leq A$ on $M\times [-2,0]$,\\
then for all $x,y\in M$ and $s,t\in [-1,0]$, we have
$$\frac{1}{B}d_s(x,y) - B \sqrt{|t-s|} \leq d_t(x,y)\leq B d_s(x,y)+B\sqrt{|t-s|}.$$
\end{thm}

\vspace{4 mm}
Let $(M^{n},(g_{t})_{t\in[-2,0)})$ be a closed solution of Ricci
flow. Standard theory (see, for example, Chapter 24, Section 2 of \cite{chowbook3}) guarantees that, for any $(x,t)\in M\times(-2,0)$,
there exists a unique fundamental solution $K(x,t;\cdot,\cdot):M\times[-2,t)\to(0,\infty)$
of the conjugate heat equation based at $(x,t)$. That is, $K(x,t;\cdot,\cdot)$
is the unique smooth function on $M\times[-2,t)$ such that 
\[
(-\partial_{s}-\Delta_{g_{s}}+R_{g_{s}})K(x,t;\cdot,\cdot)=0\hspace{6mm}\mbox{ on }M\times[-2,t),
\]
\[
\int_{M}K(x,t;y,s)f(y)dg_{s}(y)\to f(x)\hspace{6mm}\mbox{ as }s\nearrow t
\]
for any continuous $f:M\to\mathbb{R}$. Moreover, $K$ is smooth on
its domain, and if $(y,s)$ are fixed, then $K(\cdot,\cdot;y,s)$
is the fundamental solution of the heat equation: 
\[
(\partial_{t}-\Delta_{g_{t}})K(\cdot ,\cdot;y,s)=0\hspace{6mm}\mbox{ on } M\times(s,0),
\]
\[
\int_{M}K(x,t;y,s)f(x)dg_{t}(x)\to f(y)\hspace{6mm}\mbox{ as }t\searrow s.
\]
Given $(q,t)\in M\times(-2,0)$, let $u_{q,t}:M\times[-2,t)\to(0,\infty)$
be the conjugate heat kernel based at $(q,t)$, and write $u_{q,t}(y,s)=(4\pi(t-s))^{-\frac{n}{2}}e^{-f_{q,t}(y,s)}$.
The corresponding pointed entropy is defined to be $\mathcal{W}_{q,t}(\tau):=\mathcal{W}(g_{t-\tau},f_{q,t}(t-\tau),\tau)$.
Note that, if $(M^{n},g_{t})=(\mathbb{R}^{n},g_{euc})$ is the static,
flat Euclidean space, then $u_{x,t}(y,s)=(4\pi(t-s))^{-\frac{n}{2}}e^{-\frac{|x-y|^{2}}{4(t-s)}}$,
and $\mathcal{W}_{x,t}(\tau)=0$ for all $\tau>0$. Perelman's differential
Harnack inequality guarantees that $\mathcal{W}_{x,t}(\tau)\leq0$
in general, and the following $\epsilon$-regularity theorem demonstrates
that, wherever the pointed entropy is almost-Euclidean, the space-time
geometry nearby is almost-Euclidean as well.
\begin{thm} \label{epsilonreg}
(Theorem 1.16 of \cite{hein}) For any $A<\infty$, there exists $\epsilon=\epsilon(n,A)>0$
such that the following holds. Let $(M^{n},(g_{t})_{t\in[-2,0)})$
be a closed Ricci flow satisfying $\nu[g_{-2},4]\geq-A$ and $|R|(\cdot,t)\leq A|t|^{-1}$
on $M$ for all $t\in[-2,0)$. If $(q,t)\in M\times[-1,0)$ satisfies
$\mathcal{W}_{q,t}(\tau)\geq-\epsilon$, then $(r_{Rm}(q,t))^{2}\geq\epsilon\tau$. 
\end{thm}

\section{Estimates for Conjugate Heat Kernels Based at the Singular Time}

The following lemma is mostly a combination of the proofs of Theorem 1.2
in \cite{bam2} and Theorem 1.4 in \cite{BZ1}.

\begin{lem} \label{heatkernelbound}
For any $A<\infty$, there exists $C^{*}=C^{*}(A,n)<\infty$ such
that the following holds. Let $(M^{n},(g_{t})_{t\in[-2,0)})$ is a
closed solution of Ricci flow satisfying $\nu[g_{-2},4]\geq-A$ and
$|R|(x,t)\leq A|t|^{-1}$ for all $t\in[-2,0)$. Then, for any $x,y\in M$
and $-\frac{1}{2}\leq s<t<0$, we have 
\[
\dfrac{1}{C^{*}(t-s)^{\frac{n}{2}}}\exp\left(-\dfrac{C^{*}d_{s}^{2}(x,y)}{t-s}\right)\leq K(x,t;y,s)\leq\dfrac{C^{*}}{(t-s)^{\frac{n}{2}}}\exp\left(-\dfrac{d_{s}^{2}(x,y)}{C^{*}(t-s)}\right).
\]
\end{lem}

\begin{proof}
\noindent First note the reduced distance bound 
\[
\ell_{(x,t)}(x,s)\leq\dfrac{1}{2\sqrt{t-s}}\int_{0}^{t-s}\sqrt{\tau}\dfrac{A}{|t|+\tau}d\tau\leq\dfrac{1}{2\sqrt{t-s}}\int_{0}^{t-s}\dfrac{A}{\sqrt{\tau}}d\tau=A,
\]
so by Perelman's differential Harnack inequality, $K(x,t;x,s)\geq(4\pi(t-s))^{-\frac{n}{2}}e^{-A}$
for all $x\in M$ and $-2\leq s<t<0$.\\
\textbf{Claim:} There exists $C'=C'(A,n)<\infty$ such that, for $-\frac{1}{2}\leq s<0$
and $t\in(s,\frac{1}{2}s]$, $x,y\in M$, we have 
\begin{equation} \label{earlyheat}
\dfrac{1}{C'(t-s)^{\frac{n}{2}}}\exp\left(-\dfrac{C'd_{\tau}^{2}(x,y)}{t-s}\right)\leq K(x,t;y,s)\leq\dfrac{C'}{(t-s)^{\frac{n}{2}}}\exp\left(-\dfrac{d_{\tau}^{2}(x,y)}{C'(t-s)}\right),
\end{equation}
where $\tau\in\{s,t\}$.\\
 This will just follow from an appropriate rescaling and the corresponding
Gaussian bounds for Ricci flow with bounded scalar curvature. In fact,
consider the rescaled flow $\widetilde{g}_{r}:=|t|^{-1}g_{t+|t|r}$,
$r\in[-|t|^{-1}(2+t),0]$. Then $|\widetilde{R}|\leq A$ on $M\times[-2,0]$,
and 
\[
\nu[\widetilde{g}_{-2},4]=\nu[g_{3t},4|t|]\geq\nu[g_{-2},4|t|+(2-3|t|)]\geq\nu[g_{-2},2+|t|]\geq-A,
\]
so the Bamler-Zhang heat kernel estimates \cite{BZ1} give $C'=C'(A,n)<\infty$
such that for all $x,y\in M$ and $r\in[-1,0)$ we have 
\[
\dfrac{1}{C'|r|^{\frac{n}{2}}}\exp\left(-\dfrac{C'\widetilde{d}_{\tau}^{2}(x,y)}{|r|}\right)\leq\widetilde{K}(x,0;y,r)\leq\dfrac{C'}{|r|^{\frac{n}{2}}}\exp\left(-\dfrac{\widetilde{d}_{\tau}^{2}(x,y)}{C'|r|}\right),
\]
where $\tau\in\{0,r\}$. Also, we know the behavior of the heat kernel
under rescaling: $\widetilde{K}(x,0;y,r)=|t|^{\frac{n}{2}}K(x,t;y,t+|t|r)$,
so taking $r:=-|t|^{-1}(t-s)$ gives $K(x,t;y,s)=|t|^{-\frac{n}{2}}\widetilde{K}(x,0;y,-|t|^{-1}(t-s))$.
Note that $|t|^{-1}(t-s)\leq(|s|/2)^{-1}\cdot(s/2-s)=1$ because $t\in(s,s/2]$,
so the claim follows. $\square$\\
Now consider the case where $s/2<t<0$. A special case of the reproduction
formula for the heat kernel is 
\[
K(x,t;y,s)=\int_{M}K(x,t;z,\frac{1}{2}(t+s))K(z,\frac{1}{2}(t+s);y,s)dg_{\frac{1}{2}(t+s)}(z)
\]
for all $x,y\in M$ and $-2\leq s<0$ and $t\in(s/2,0)$. Also, the
above claim implies 
\[
K(z,\frac{1}{2}(t+s);y,s)\leq\dfrac{2^{\frac{n}{2}}C'}{(t-s)^{\frac{n}{2}}},
\]
so combining this with the reproduction formula gives 
\[
K(x,t;y,s)\leq\dfrac{2^{\frac{n}{2}}C'}{(t-s)^{\frac{n}{2}}}\int_{M}K(x,t;z,\frac{1}{2}(t+s))dg_{\frac{1}{2}(t+s)}(z)=\frac{2^{\frac{n}{2}}C'}{(t-s)^{\frac{n}{2}}}.
\]
Set $c_{0}:=(4\pi)^{-\frac{n}{2}}e^{-A}$, so that the above claim
gives $D=D(A,n)<\infty$ such that, for any $x,y\in M$ with $d_{\frac{1}{2}(t+s)}(x,y)\geq D\sqrt{t-s}$,
we have $$K(x,\frac{1}{2}(t+s);y,s)\leq\dfrac{c_{0}}{2(t-s)^{\frac{n}{2}}}.$$
Thus, for any $x\in M$, we have 
\begin{align*}
\dfrac{c_{0}}{(t-s)^{\frac{n}{2}}} & \leq K(x,t;x,s)=\int_{M}K(x,t;y,\frac{1}{2}(t+s))K(y,\frac{1}{2}(t+s);x,s)dg_{\frac{1}{2}(t+s)}(y)\\
 & \leq\dfrac{2^{\frac{n}{2}}C'}{(t-s)^{\frac{n}{2}}}\int_{B(x,\frac{1}{2}(t+s),D\sqrt{t-s})}K(x,t;y,\frac{1}{2}(t+s))dg_{\frac{1}{2}(t+s)}(y)+\dfrac{c_{0}}{2(t-s)^{\frac{n}{2}}},
\end{align*}
so applying the upper bound of (\ref{earlyheat}) with $\tau=\frac{1}{2}(t+s)$ gives
\begin{equation} \label{intlowerbound}
\int_{B(x,\frac{1}{2}(t+s),D\sqrt{t-s})}K(x,t;y,\frac{1}{2}(t+s))dg_{\frac{1}{2}(t+s)}(y)\geq c_{0}2^{-\frac{(n+2)}{2}}(C')^{-1}=:c_{0}'.
\end{equation}

Now consider the rescaled flow $\widehat{g}_{r}:=(t-s)^{-1}g_{(t-s)r+\frac{1}{2}(s+t)}$,
which satisfies $$|\widehat{R}|\leq\dfrac{2A}{|s+t|}(t-s)\leq2A$$ 
for $r\in[-2,0]$, and
\[
\nu[\widehat{g}_{-2},4]=\nu[g_{\frac{1}{2}(s+t)-2(t-s)},4(t-s)]\geq\nu[g_{-2},4(t-s)+\frac{1}{2}(s+t)-2(t-s)+2]\geq\nu[g_{-2},4]
\]
since $-\frac{1}{2}\leq s\leq t<0$. We can thus apply Theorem \ref{distortion} to obtain $B=B(A,n)<\infty$ such that
$$\frac{1}{B}\widehat{d}_{r_1}(x_1,x_2) -B\sqrt{|r_2-r_1|}\leq \widehat{d}_{r_2}(x_1,x_2)\leq B\widehat{d}_{r_2}(x_1,x_2)+B\sqrt{|r_2-r_1|}$$
for all $r_1,r_2\in [-\frac{1}{2},0]$ and $x_1,x_2 \in M$. In terms of the unrescaled flow, this gives 
\begin{equation} \label{rescaleddist} \frac{1}{B}d_{t_1}(x_1,x_2) - B\sqrt{|t_2-t_1|} \leq d_{t_2}(x_1,x_2)\leq Bd_{t_1}(x_1,x_2)+B\sqrt{|t_2-t_1|} \end{equation}
for all $t_1,t_2 \in [s,\frac{1}{2}(s+t)]$ and $x_1,x_2\in M$. For any $x,y\in M$, we combine (\ref{earlyheat}),(\ref{intlowerbound}),(\ref{rescaleddist}) to obtain 
\begin{align*}
K(x,t;y,s) & \geq\int_{B(x,\frac{1}{2}(t+s),D\sqrt{t-s})}K(x,t;z,\frac{1}{2}(t+s))K(z,\frac{1}{2}(t+s);y,s)dg_{\frac{1}{2}(t+s)}(z)\\
 & \geq\int_{B(x,\frac{1}{2}(t+s),D\sqrt{t-s})}K(x,t;z,\frac{1}{2}(t+s))\\&\hspace{10 mm}\times
 \dfrac{1}{C'(t-s)^{\frac{n}{2}}}\exp\left(-\dfrac{2C'}{t-s}(d_{\frac{1}{2}(s+t)}(x,y)+D\sqrt{t-s})^{2}\right)dg_{\frac{1}{2}(t+s)}(z)\\
 & \geq\dfrac{c_{0}'}{C'(t-s)^{\frac{n}{2}}}e^{-4C'D^{2}}\exp\left(-\dfrac{4C'}{t-s}\left( Bd_s(x,y)+B\sqrt{|t-s|} \right)^2 \right)
 \\ &\geq \frac{c_0'}{C'(t-s)^{\frac{n}{2}}}e^{-4C'D^2-8C'B^2}\exp \left( -\frac{8C'B^2 d_s^2(x,y)}{t-s} \right)
\end{align*}
This and (\ref{earlyheat}) give $C^{*}(A,n)<0$ such that, for all $x,y\in M$
and $-\frac{1}{2}\leq s<t<0$, we have 
\[
K(x,t;y,s)\geq\dfrac{1}{C^{*}(t-s)^{\frac{n}{2}}}\exp\left(-\dfrac{C^{*}d_{s}^{2}(x,y)}{t-s}\right).
\]
In particular, for any $r\in[s,\frac{1}{2}(s+t)]$, we have 
\[
\int_{B(x,r,\sqrt{t-r})}K(x,t;y,r)dg_{r}(y)\geq\dfrac{e^{-C^{*}}}{C^{*}(t-s)^{\frac{n}{2}}}|B(x,r,\sqrt{t-r})|_{r}.
\]
Applying Theorem \ref{volumebounds} to the rescaled flow gives
\[
|\widehat{B}(y,r,\frac{1}{\sqrt{2}})|_{\widetilde{g}_r}\geq b\hspace{6mm}\mbox{ for all }\hspace{6mm}r\in[-\frac{1}{2},0],
\]
where $b=b(A,n)>0$. Thus, for any $r\in[s,\frac{1}{2}(s+t)]$, we have
\begin{align*}
|B(x,r,\sqrt{t-r})|_{g_r}&\geq |B(x,r,\frac{1}{\sqrt{2}}\sqrt{t-s})|_{g_r}=(t-s)^{\frac{n}{2}}|\widehat{B}(y,\tau,\frac{1}{\sqrt{2}})|_{\widetilde{g}_r}\\&\geq b(t-s)^{\frac{n}{2}},
\end{align*}
where $\tau$ is defined by $\frac{1}{2}(s+t)+(t-s)\tau=r$, so that
$\tau\in[-\frac{1}{2},0]$. Combining estimates gives
\begin{equation} \label{lowerconcentrate} \int_{B(x,r,\sqrt{t-r})} K(x,t;y,r)dg_r(y) \geq b(C^{\ast})^{-1}e^{-C^{\ast}}=:c_{\ast} \end{equation}
for all $r\in [s,\frac{1}{2}(s+t)]$.

\noindent \textbf{Case 1:} For any $y\in M$ with $d_{s}(x,y)\geq 4B^2\sqrt{t-s}$, the distortion estimate (\ref{rescaleddist}) gives
$$d_r(x,y)\geq \frac{1}{B}d_s(x,y)-B\sqrt{r-s} \geq \frac{1}{2B}d_s(x,y)$$
for all $r\in [s,\frac{1}{2}(s+t)]$. Then the Hein-Naber concentration
inequality (Theorem 1.30 of \cite{hein}) gives 
\begin{align*}
\left(\int_{B(x,r,\sqrt{t-r})}K(x,t;z,r)dg_{r}(z)\right) \hspace{ -20 mm}& \hspace{20 mm} \left(\int_{B(y,r,\sqrt{t-r})}K(x,t;z,r)dg_{r}(z)\right) \\& \leq\exp\left(-\dfrac{(d_{r}(x,y)-2\sqrt{t-r})^{2}}{8(t-r)}\right)\\
 & \leq \exp \left(- \frac{1}{8(t-r)}\left( \frac{1}{B}d_s(x,y)-2\sqrt{t-s} \right)^2 \right) \\ &\leq \exp \left( -\frac{d_s^2(x,y)}{32B^2(t-s)} \right).
\end{align*}
Combining this with \ref{lowerconcentrate} gives 
\begin{align*}
\int_{B(y,r,\sqrt{t-r})}K(x,t;z,r)dg_{r}(z)&\leq c_{\ast}^{-1} \exp\left(-\dfrac{d_{s}^{2}(x,y)}{32B^{2}(t-s)}\right) \end{align*}
We integrate from $r=s$ to $r=\frac{1}{2}(t+s)$ to get
\[
\int_{Q^{+}(y,s,\sqrt{\frac{1}{2}(t-s)})}K(x,t;z,r)dg_{r}(z)dr\leq c_{\ast}^{-1}(t-s)\exp\left(-\dfrac{d_{s}^{2}(x,y)}{32B^2(t-s)}\right).
\]
We now combine this with the on-diagonal upper bound, obtaining $\overline{C}=\overline{C}(A,n)<\infty$ such that
\[
\int_{Q^{+}(y,s,\sqrt{\frac{1}{2}(t-s)})}K^{2}(x,t;z,r)dg_{r}(z)dr\leq\dfrac{\overline{C}}{(t-s)^{\frac{n}{2}-1}}\exp\left(-\dfrac{d_{s}^{2}(x,y)}{\overline{C}(t-s)}\right).
\]
In terms of the rescaled flow, this is 
\[
\int_{\widehat{Q}^{+}(y,-\frac{1}{2},\frac{1}{\sqrt{2}})}\widehat{K}^{2}(x,1/2;z,r)d\widehat{g}_{r}(z)dr\leq\overline{C}\exp\left(-\dfrac{d_{s}^{2}(x,y)}{\overline{C}(t-s)}\right),
\]
The Bamler-Zhang parabolic mean value inequality for solutions to
the conjugate heat equation (Lemma 4.2 in \cite{BZ1}) applied to
the rescaled flow (on the time interval $[-\frac{1}{2},0]$) gives 
\[
\widehat{K}^{2}(x,\frac{1}{2};y,-\frac{1}{2})\leq C''\int_{\widehat{Q}^{+}(y,-\frac{1}{2},\frac{1}{\sqrt{2}})}\widehat{K}^{2}(x,\frac{1}{2};z,r)d\widehat{g}_{u}(z)du\leq C''\overline{C}\exp\left(-\dfrac{d_{s}^{2}(x,y)}{\overline{C}(t-s)}\right),
\]
for some $C''=C''(A,n)<\infty$, so rescaling back gives 
\begin{align*}
K^{2}(x,t;y,s) \leq & \frac{C''\overline{C}}{(t-s)^n}\exp\left( -\frac{d_s^2(x,y)}{\overline{C}(t-s)} \right).
\end{align*}

\noindent \textbf{Case 2:} If instead $d_{s}(x,y)\leq4B^2\sqrt{t-s}$,
then 
\[
K(x,t;y,s)\leq\dfrac{2^{\frac{n}{2}}C'}{(t-s)^{\frac{n}{2}}}\leq\dfrac{e2^{\frac{n}{2}}C'}{(t-s)^{\frac{n}{2}}}\exp\left(-\dfrac{d_{s}^{2}(x,y)}{16B^{2}(t-s)}\right).
\]
\end{proof}

\vspace{4 mm}
\noindent Throughout this section, let $u_{q,t}$ be the conjugate
heat kernel based at $(q,t)$, and write $u_{q,t}(x,s)=(4\pi(t-s))^{-\frac{n}{2}}e^{-f_{q,t}(x,s)}.$ The following lemma is essentially obtained by passing Lemma \ref{heatkernelbound} to the limit as $t \nearrow 0$, and extends Propositions 2.7 and 2.8 of \cite{mant}.\\

\begin{lem} \label{heatsing}
Let $(M^{n},(g_{t})_{t\in[-2,0)},q)$ be a closed, pointed Ricci flow
with $\nu[g_{-2},4]\geq-A$ and $|R(x,t)|\leq A|t|^{-1}$ for all
$(x,t)\in M\times[-2,0)$. Also suppose $t_{i}\nearrow0$ and $q_{i}\to q$
in $M$. Then there is some subsequence of $(u_{q_{i},t_{i}})_{i\in\mathbb{N}}$
which converges in $C_{loc}^{\infty}(M\times(-1,0))$ to some $u_{q,0}\in C^{\infty}(M\times[-1,0))$
satisfying $\int_{M}u_{q,0}(x,t)dg_{t}(x)=1$ for $t\in[-1,0)$, as
well as the conjugate heat equation $(-\partial_{t}-\Delta+R)u_{q,0}=0$.
In addition, there exists $C=C(A,n)<\infty$ such that 
\[
\dfrac{1}{C|s|^{\frac{n}{2}}}\exp\left(-\dfrac{Cd_{s}^{2}(y,q)}{|s|}\right)\leq u_{q,0}(y,s)\leq\dfrac{C}{|s|^{\frac{n}{2}}}\exp\left(-\dfrac{d_{s}^{2}(y,q)}{C|s|}\right)
\]
for all $(y,s)\in M\times[-1,0)$. 
\end{lem}

\begin{proof}
For any closed solution of Ricci flow, a subsequence of $u_{q_{i},t_{i}}$ must converge
in $C_{loc}^{\infty}(M\times[-2,0))$ to some $u_{q,0}$ solving the
conjugate heat equation on $M\times[-2,0)$, as shown in \cite{mant}.
Since $M$ is closed, $\int_{M}u_{q,0}(x,t)dg_{t}(x)=1$ is immediate,
so it suffices to prove the Gaussian bounds for any limit $u_{q,0}$.
Fix $\alpha\in(0,1]$, and let $i_{0}\in\mathbb{N}$ be sufficiently
large so that $t_{i}-\alpha\geq\frac{1}{2}\alpha$ for all $i\geq i_{0}$.
By the previously established heat kernel bounds, there exists $C^{\ast}=C^{\ast}(A,n)<\infty$
such that, for all $(y,s)\in M\times[-1,-\alpha]$ and $i\geq i_{0}$,
we have 
\begin{align*}
u_{q_{i},t_{i}}(y,s) & \geq\dfrac{1}{C^{\ast}(t_{i}-s)^{\frac{n}{2}}}\exp\left(-\dfrac{2C^{\ast}(d_{s}^{2}(q_{i},q)+d_{s}^{2}(q,y))}{t_{i}-s}\right)\\
 & \geq\dfrac{1}{C^{\ast}|s|^{\frac{n}{2}}}\exp\left(-\dfrac{2C^{\ast}d_{s}^{2}(q_{i},q)}{\frac{1}{2}\alpha}\right)\exp\left(-\dfrac{2C^{\ast}d_{s}^{2}(q,y)}{\frac{1}{2}|s|}\right).
\end{align*}
Note that $d_{s}(q_{i},q)\to0$ uniformly in $s\in[-1,-\alpha]$ as
$i\to\infty$, so for any $(y,s)\in M\times[-1,-\alpha]$, 
\[
u_{q,0}(y,s)=\lim_{i\to\infty}u_{q_{i},t_{i}}(y,s)\geq\dfrac{1}{C^{\ast}|s|^{\frac{n}{2}}}\exp\left(-\dfrac{4C^{\ast}d_{s}^{2}(q,y)}{|s|}\right).
\]
Similarly, for any $(y,s)\in M\times[-1,-\alpha]$ and $i\geq i_{0}$,
we have (since $(a-b)_{+}^{2}\geq\frac{1}{2}a^{2}-b^{2}$ for $a,b>0$)
\begin{align*}
u_{q_{i},t_{i}}(y,s)&\leq\dfrac{C^{\ast}}{(t_{i}-s)^{\frac{n}{2}}}\exp\left(\dfrac{-\frac{1}{2}d_{s}^{2}(q,y)+d_{s}^{2}(q,q_{i})}{C^{\ast}(t_{i}-s)}\right)\\ &\leq\dfrac{2^{\frac{n}{2}}C^{\ast}}{|s|^{\frac{n}{2}}}\exp\left(\dfrac{d_{s}^{2}(q,q_{i})}{\frac{1}{2}C^{\ast}|s|}\right)\exp\left(-\dfrac{d_{s}^{2}(q,y)}{2C^{\ast}|s|}\right),
\end{align*}
so the claim follows as for the lower bound. 
\end{proof}

\begin{defn}
Any limit $u_{q,0}$ as in the statement of Lemma 10 is called a conjugate heat kernel based
at the singular time. The set of
such functions $u_{q,0}$ is denoted $\mathcal{U}_{q}$, as in \cite{mant}.
\end{defn}
Note that we are not able to establish the uniqueness
of $u_{q,0}$ given a point $q\in M$ (in fact, this is not even known
under assumption \eqref{eq:curv}), but the collection of
such functions satisfies strong compactness properties. By the uniform Gaussian estimates and parabolic regularity on compact
subsets of $M\times(-1,0)$, $\mathcal{U}_q$ is compact in $C_{loc}^{\infty}$.
Let $\mathcal{F}_{q}$ be the set of $f_{q,0}\in C^{\infty}(M\times(-1,0))$,
where $u_{q,0}(x,t)=(4\pi|t|)^{-\frac{n}{2}}e^{-f_{q,0}(x,t)}$. By
the locally uniform bounds on $u_{q}\in\mathcal{U}_{q}$ and their
derivatives, we observe that $\mathcal{\mathcal{F}}_{q}$ is also
compact in $C_{loc}^{\infty}$. Thus Perelman's differential Harnack
inequality passes to the limit to give 
\[
\tau(R+2\Delta f_{q}-|\nabla f_{q}|^{2})+f_{q}-n\leq0\hspace{6mm}\mbox{ on }\hspace{6mm}M\times(-1,0)
\]
for any $f_{q}\in\mathcal{F}_{q}$, where $\tau:=|t|$. As in \cite{mant}, we also define 
\[
\theta_{q}(t):=\inf_{f\in\mathcal{F}_{q}}\mathcal{W}(g_{t},f(t),\tau(t)).
\]
Because $\mathcal{F}_{q}$ is compact in $C_{loc}^{\infty}$, and
because $\theta_{q}(t)\geq\mu[g_{t},\tau(t)]>-\infty$, this infimum
is actually achieved at any $t\in(-1,0)$ by some $f_{t}\in\mathcal{F}_{q}$. For $-1<s<t<0$, Perelman's entropy monotonicity
gives 
\[
\theta_{q}(s)\leq\mathcal{W}(g_{s},f_{t}(s),\tau(s))\leq\mathcal{W}(g_{t},f_{t}(t),\tau(t))=\theta_{q}(t),
\]
so $\theta_{q}$ is nondecreasing. Similar reasoning gives
\[
0\leq\theta_{q}(t)-\theta_{q}(s)\leq\int_{s}^{t}2|r|\int_{M}\left|Rc_{g_{r}}+\nabla^{2}f_{s}(r)-\dfrac{g_{r}}{2|r|}\right|^{2}\dfrac{e^{-f_{s}(r)}}{(4\pi|r|)^{\frac{n}{2}}}dg_{r}dr,
\]
but the integrand is bounded on any compact subset of $M\times(-1,0)$,
by the uniform estimates for $f\in\mathcal{F}_{q}$. Thus $\theta_{q}$
is locally Lipschitz. Moreover, $\theta_{q}(t)\leq0$ for all $t\in(-1,0)$
by Perelman's Harnack inequality, so we can define the heat kernel density function  $$\Theta(q):=\lim_{t\nearrow0}\theta_{q}(t).$$

Fix a sequence $t_{i}\nearrow0$, and consider the rescaled flows
$\widetilde{g}_{t}^{i}:=|t_{i}|^{-1}g_{t_{i}+|t_{i}|t}$, $t\in[-2,0]$,
and $\widetilde{f}_{i}(t):=f_{t_{i}}(t_{i}+|t_{i}|t)$. By the monotonicity
of $\theta_{q}$, we have $\lim_{i\to\infty}(\theta_{q}(t_{i})-\theta_{q}(t_{i}-\rho|t_{i}|))=0$
for any fixed $\rho>0$. Since 
\begin{align*}
0 & \leq\mathcal{W}(\widetilde{g}_{0}^{i},\widetilde{f}_{i}(0),1)-\mathcal{W}(\widetilde{g}_{-\rho}^{i},\widetilde{f}_{i}(-\rho),\tau(-\rho))\\
 & =\mathcal{W}(g_{t_{i}},f_{t_{i}}(t_{i}),|t_{i}|)-\mathcal{W}(g_{t_{i}-\rho|t_{i}|},f_{t_{i}}(t_{i}-\rho|t_{i}|),|t_{i}|+\rho|t_{i}|)\\
 & \leq\theta_{q}(t_{i})-\theta_{q}(t_{0}-\rho|t_{i}|),
\end{align*}
we may conclude that 
\begin{align} \label{eq:entmoot}
0 & =\lim_{i\to\infty}\mathcal{W}(\widetilde{g}_{0}^{i},\widetilde{f}_{i}(0),1)-\mathcal{W}(\widetilde{g}_{-\rho}^{i},\widetilde{f}_{i}(-\rho),1+\rho)\\
 & =\lim_{i\to\infty}\int_{-\rho}^{0}2(1+|t|)\int_{M}\left|Rc_{\widetilde{g}_{t}^{i}}+\nabla^{2}\widetilde{f}_{i}(t)-\dfrac{\widetilde{g}_t^i}{2(1+|t|)}\right|^{2}\dfrac{e^{-\widetilde{g}_t^i}}{(4\pi(1+|t|))^{-\frac{n}{2}}}d\widetilde{g}_{t}^{i}dt.
\end{align}

By the argument of Theorem 1.2 of \cite{bam2}, we know that, after
passing to a subsequence, $(M,\widetilde{g}_{0}^{i},q)$ converge
to a singular shrinking GRS, and $\widetilde{f}_{i}(0)$ converge
to the corresponding potential function. The only difference is that,
in \cite{bam2}, the soliton potential function is obtained from
limiting a fixed conjugate heat kernel based at the singular time,
whereas we are obtaining a soliton potential function from a sequence
in $\mathcal{F}_{q}$. The proof is almost exactly the same, since
the estimates for elements of $\mathcal{F}_{q}$ are uniform, but
we rewrite the relevant parts of the argument in \cite{bam2} here
for completeness, and because we would like to pass the heat kernel bounds of Lemma \ref{heatsing} to the limit.

By Bamler's compactness theorem (\thref{bamthm}), we can pass to
a subsequence so that $(M,\widetilde{g}_{0}^{i},q)$ converge to a
pointed singular space $(\mathcal{X},q_{\infty})=(X,d,\mathcal{R},g,q_{\infty})$,
with associated convergence scheme $\Phi_{i}:U_{i}\to V_{i}$. For
any $x\in\mathcal{R}$, we have $r:=r_{Rm}^{X}(x)>0$, so by Proposition
4.1 in \cite{bam2}, $r_{Rm}^{\widetilde{g}^i}(\Phi_{i}(x),0)>\frac{r}{2}$
for sufficiently large $i\in\mathbb{N}$. Because $|R_{\widetilde{g}_{0}^{i}}|\leq A$
and $\nu[\widetilde{g}_{-2}^i,4]\geq-A$, backwards pseudolocality
(Theorem 1.5 in \cite{BZ1}) gives $\alpha=c(n,A)>0$ such that $r_{Rm}(y,s)>\alpha r$
for all $(y,s)\in B_{\widetilde{g}^{i}}(\Phi_{i}(x),0,\alpha r)\times[-\alpha^{2}r^{2},0]$
for all $i\in\mathbb{N}$. By \thref{heatsing}, we have the uniform bounds
\[
\dfrac{1}{C}\exp\left(-\dfrac{Cd_{s}^{2}(y,q)}{|s|}\right)\leq e^{-f_{t_{i}}(y,s)}\leq C\exp\left(-\dfrac{d_{s}^{2}(y,q)}{C|s|}\right)
\]
for all $(y,s)\in[-1,0)$. This implies
\[
-\log C+\dfrac{\widetilde{d}_{t}^{2}(y,q)}{C(1-t)}\leq\widetilde{f}_{i}(t)\leq\log C+C\dfrac{\widetilde{d}_{t}^{2}(y,q)}{(1-t)}
\]
for all $t\in[-2,0]$, hence (because $\Phi_{i}$ is a $\epsilon_{i}$-Gromov-Hausdorff
map for some sequence $\epsilon_{i}\to0$)
\[ \label{fboundssequential}
-\log C^{\ast}+\dfrac{d^{2}(q_{\infty},x)}{C^{\ast}}\leq\widetilde{f}_{i}(\Phi_{i}(x),0)\leq\log C^{\ast}+C^{\ast}d^{2}(q_{\infty},x)
\]
for all $x\in U_{i}\cap B^{X}(q_{\infty},D_{i})$, where $D_{i}\to\infty$.
By parabolic regularity theory applied to $\widetilde{u}_{i}(t):=(4\pi(1-t))^{-\frac{n}{2}}e^{-\widetilde{f}_{i}(t)}$
on $B_{\widetilde{g}^{i}}(\Phi_{i}(x),0,\alpha r)\times[-\alpha^{2}r^{2},0]$,
we find that 
$$\limsup_{i\to\infty}\sup_{B_{\widetilde{g}^{i}}(\Phi_{i}(x),0,\frac{1}{2}\alpha r)\times[-\frac{1}{2}\alpha^{2}r^{2},0]}|\nabla^{k}\widetilde{u}_{i}|_{\widetilde{g}^i}>0$$
for all $k\in\mathbb{N}$. Along with the locally uniform upper bound
for $\widetilde{f}_{i}$, we get similar bounds for $\widetilde{f}_{i}$,
so that we can pass to a subsequence such that $\widetilde{f}_{i}(0)$
converges in $C_{loc}^{\infty}$ to some $f_{\infty}\in C^{\infty}(\mathcal{R})$.
Suppose by way of contradiction that there exists $x^{*}\in\mathcal{R}$
such that 
\[
\left|Rc_{g_{\infty}}+\nabla^{2}f_{\infty}-\dfrac{g_{\infty}}{2}\right|^{2}(x^{*})\geq c_{0}>0.
\]
Then this quantity is at least $\frac{1}{2}c_{0}$ on some ball $B^{X}(x^{*},r)\subseteq\mathcal{R}$,
so for $x\in B_{\widetilde{g}_0^i}(\Phi_{i}(x^{*}),\frac{1}{2}r)$
and sufficiently large $i$, we have 
\[
\left|Rc_{\widetilde{g}_{0}^{i}}+\nabla^{2}\widetilde{f}_{i}(0)-\dfrac{\widetilde{g}_{0}^{i}}{2}\right|^{2}(x)\geq\frac{c_{0}}{4}.
\]
However, this along with backwards pseudolocality and parabolic regularity
give 
\[
\left|Rc_{\widetilde{g}_{t}^{i}}+\nabla^{2}\widetilde{f}_{i}(t)-\dfrac{\widetilde{g}_{t}^{i}}{2(1+|t|)}\right|^{2}(x)\geq\delta
\]
for $(x,t)\in B_{\widetilde{g}^{i}}(\phi_{i}(x^{\ast}),0,\delta)\times[-\delta^{2},0]$,
where $\delta>0$ is small, (depending on $x^{\ast}$ but not on $i$)
contradicting \eqref{eq:entmoot}. The estimate (\ref{fboundssequential}) passes to the limit to give
\begin{equation} \label{fboundslimit} -\log C^{\ast} +\frac{d^2(q_{\infty},x)}{C^{\ast}}\leq f_{\infty}(x)\leq \log C^{\ast} +C^{\ast}d^2(q_{\infty},x)\end{equation}
for all $x\in \mathcal{R}$.\\

\section{Integration by Parts on the Singular Ricci Soliton}

\vspace{4mm}
Now let $(X,d,\mathcal{R},g_{\infty},f_{\infty})$ be a singular shrinking
GRS as obtained in the previous section.
\begin{lem}
There exists $T=T(A,n)<\infty$ such that, for all $r>0$, we have
\[
|B^{X}(q_{\infty},r)\cap\mathcal{R}|\leq Tr^{n}.
\]
\end{lem}

\begin{proof}
By Proposition 6, we obtain $C=C(A)<\infty$ such that $|B(x,t,r)|_{t}\leq Cr^{n}$
for all $r\in(0,1]$ and $(x,t)\in M\times[-1,0)$. For the rescaled
flows $\widetilde{g}_t^i:=|t_{i}|^{-1}g_{t_i+|t_i|t}$, this
means that $|B_{\widetilde{g}^i}(x,t,r)|_{\widetilde{g}_t^i}\leq Cr^{n}$
for all $r\in(0,|t_{i}|^{-\frac{1}{2}})$ and $(x,t)\in M\times[-2|t_{i}|^{-1},0)$.
Now let $(U_{i},V_{i},\Phi_{i})$ be a convergence scheme for the
convergence $(M,\widetilde{g}_0^i,q)\to(\mathcal{X},q_{\infty})$.
Let $K$ be any compact subset of $B^{X}(q_{\infty},r)\cap\mathcal{R}$.
Then, for sufficiently large $i\in\mathbb{N}$, we have $K\subseteq U_{i}$
and 
\[
|K|\leq 2|\varphi_{i}(K)|_{\widetilde{g}_0^i}\leq2|B_{\widetilde{g}^i}(q,0,2r)|_{\tilde{g}_0^i}\leq2^{n+1}Cr^{n}.
\]
Since $K$ was arbitrary, this means $|B^{X}(q_{\infty},r)\cap\mathcal{R}|\leq2^{n+1}Cr^{n}$. 
\end{proof}
\begin{defn}The shrinker entropy of the singular shrinking GRS $(X,d,\mathcal{R},g_{\infty},f_{\infty})$
is
\[
\mathcal{W}(g_{\infty},f_{\infty}):=\int_{\mathcal{R}}(R_{g_{\infty}}+|\nabla f_{\infty}|^{2}+f_{\infty}-n)(4\pi)^{-\frac{n}{2}}e^{-f_{\infty}}dg_{\infty}.
\]
\end{defn}
This integral is finite by the previous lemma, since $|R_{\infty}|$
is bounded, $f_{\infty}$ has quadratic growth, and $|\nabla f_{\infty}|^2 \leq R+|\nabla f_{\infty}|^2 =f_{\infty}-C$ for some constant $C \in \mathbb{R}$.

In order to prove convergence of entropy, it is essential to use Perelman's
differential Harnack inequality, so that the entropy can be rewritten
as the integral of a nonpositive quantity. However, it is then necessary
to prove that the integration by parts formula 
\[
\int_{\mathcal{R}}\Delta f_{\infty}e^{-f_{\infty}}dg_{\infty}=\int_{\mathcal{R}}|\nabla f_{\infty}|^{2}e^{-f_{\infty}}dg_{\infty}
\]
holds in the singular case. This is equivalent to showing that 
\[
\int_{\mathcal{R}}\mbox{div}(\nabla e^{-f_{\infty}})dg_{\infty}=0.
\]
To this end, we recall the following integration by parts formula
\begin{lem}
(\cite{bam1} Prop 5.2) Let $\mathcal{X}=(X,d,\mathcal{R},g)$ be
a singular space with singularities of codimension $p_{0}>2$, and
$Z$ a $C^{1}$ vector field on $\mathcal{R}$ that vanishes on $\mathcal{R}\setminus B(x,r)$
for some large $r>0$. Assume there is a constant $C<\infty$ such
that 
\[
|Z|<Cr_{Rm}^{-1}\hspace{6mm}\mbox{ and }\hspace{6mm}|\operatorname{div}(Z)|<Cr_{Rm}^{-2}\hspace{6mm}\mbox{ on }\hspace{6mm}B(x,r)\cap\mathcal{R}.
\]
Then 
\[
\int_{\mathcal{R}}(\operatorname{div}Z)dg=0.
\]
\end{lem}

The hypotheses of this lemma will follow from various identities for soliton potential functions.

\begin{lem}
$\int_{\mathcal{R}}\Delta f_{\infty}e^{-f_{\infty}}dg_{\infty}=\int_{\mathcal{R}}|\nabla f_{\infty}|_{g_{\infty}}^{2}e^{-f_{\infty}}dg_{\infty}$.
\end{lem}

\begin{proof}
\noindent Now fix $r>0$, and let $\phi\in C^{\infty}(\mathcal{R})$
be a smoothing of a radial function, chosen such that $\phi|B(q_{\infty},r)=1$,
$0\leq\phi\leq1$, $|\nabla\phi|\leq 4$, and $\text{supp}(\phi)\subseteq B^{X}(q_{\infty},r+1).$
We want to apply the previous lemma to $Z:=\phi\nabla e^{-f_{\infty}}$.
Note that $R_{g_{\infty}} \geq0$ since $R_{g_t}$ is uniformly bounded
below. Also, the bound $|R|(x,t)\leq A|t|^{-1}$ passes to the limit
to give $R_{g_{\infty}}\leq A$. We know that $R_{g_{\infty}}+|\nabla f_{\infty}|^{2}-f_{\infty}=C$
for some constant $C\in\mathbb{R}$. For the purpose of this section,
we may assume that $C=0$, so that $f_{\infty}\geq0$ and $|\nabla f_{\infty}|^{2}\leq f_{\infty}$.
Also, $R_{g_{\infty}}+\Delta f_{\infty}=\frac{n}{2}$ implies that $|\Delta f_{\infty}|\leq A+\frac{n}{2}$. The quadratic growth estimates (\ref{fboundslimit}) give
\[
|Z(x)|\leq|\nabla f_{\infty}|e^{-f_{\infty}}\leq\left(\log C^{*}+C^{*}d^{2}(x,q_{\infty})\right)^{\frac{1}{2}}\exp\left(\log C^{*}-\frac{1}{C^{*}}d^{2}(q_{\infty},x)\right),
\]
\begin{align*}
|\operatorname{div}(Z(x))| & \leq2(|\nabla f_{\infty}|+|\nabla f_{\infty}|^{2}+|\Delta f_{\infty}|)e^{-f_{\infty}}\\
 & \leq4(\log C^{*}+C^{*}d^{2}(x,q_{\infty})+A+\frac{n}{2})\exp\left(\log C^{*}-\frac{1}{C^{*}}d^{2}(q_{\infty},x)\right),
\end{align*}
for $x\in\mathcal{R}.$ Both of these terms are locally bounded on
$\mathcal{R}$, so we may apply the previous lemma to $Z$ to obtain
$0=\int_{\mathcal{R}}\mbox{div}(\phi\nabla e^{-f_{\infty}})dg_{\infty}$.
Using the volume upper bound, we can conclude
\begin{align*}
\int_{\mathcal{R}}|\nabla\phi|\cdot|\nabla f_{\infty}|e^{-f_{\infty}}dg_{\infty}\leq & C(n)\int_{\mathcal{R}\cap(B^{X}(q_{\infty},r+1)\setminus B^{X}(q_{\infty},r))}|\nabla f_{\infty}|e^{-f_{\infty}}dg_{\infty}\\
\leq & C(n)\int_{\mathcal{R}\cap(B^{X}(q_{\infty},r+1)\setminus B^{X}(q_{\infty},r))}e^{-\frac{1}{2}f_{\infty}}dg_{\infty}\\
\leq & C(n,A)r^{n}\exp\left(-\frac{r^{2}}{C(n,A)}\right).
\end{align*}
The claim then follows by taking $r\to\infty$, and using the dominated
convergence theorem.
\end{proof}
\begin{cor}
\noindent The soliton entropy can also be expressed as
\[
\mathcal{W}(g_{\infty},f_{\infty})=(4\pi)^{-\frac{n}{2}}\int_{\mathcal{R}}(R_{g_{\infty}}+2\Delta f_{\infty}-|\nabla f_{\infty}|^{2}+f_{\infty}-n)dg_{\infty},
\]
which has nonpositive integrand by passing Perelman's differential Harnack inequality
to the limit.
\end{cor}

\section{Proof of Entropy Convergence}
\begin{thm} \thlabel{entcon} Suppose $(M^n,(g_t)_{t\in [-2,0)},q)$ is a closed, pointed solution of Ricci flow satisfying $\nu[g_{-2},4]\geq -A$ and
$$|R(\cdot,t)|\leq \frac{A}{|t|}$$
for all $t\in [-2,0)$. Let $(\mathcal{X},q_{\infty})=(X,d,\mathcal{R},g_{\infty},q_{\infty})$ be a singular space obtained as a pointed limit of $(M,\widetilde{g}_0^i,q)$, where $t_i \nearrow 0$, and $\widetilde{g}_t^i:=|t_i|^{-1}g_{t_i+|t_i|t}$. Also assume $f_{\infty}\in C^{\infty}(\mathcal{R})$ is obtained by limiting $\widetilde{f}_i(0)$ as in Section 3, where $f_i(t):=f_{t_i}(t_i +|t_i|t)$, and $f_{t_i}\in \mathcal{F}_q$ satisfy $\theta_q(t_i)=\mathcal{W}(g_{t_i},f_{t_i},|t_i|)$. Then
\[
\Theta(q)=\lim_{i\to\infty}\mathcal{W}(\widetilde{g}_0^i,\widetilde{f}_{i}(0),1)=\mathcal{W}(g_{\infty},f_{\infty}).
\]
\end{thm}

\begin{proof}
The first equality is by definition. Let $(U_{i},V_{i},\Phi_{i})$
be the convergence scheme for $(M,\widetilde{g}_0^i,q)\to(\mathcal{X},q_{\infty})$.
Then, for any compact subset $K\subseteq\mathcal{R}$, we have for
large enough $i\in\mathbb{N}$ that 
\begin{align*}
\int_{K}(R_{g_{\infty}}+2\Delta f_{\infty}- & |\nabla f_{\infty}|^{2}+f_{\infty}-n)(4\pi)^{-\frac{n}{2}}dg_{\infty}\\
 & =\lim_{i\to\infty}\int_{\Phi_{i}(K)}(R_{\widetilde{g}_0^i}+2\Delta\widetilde{f}_{i}(0)-|\nabla\widetilde{f}_{i}(0)|^{2}+\widetilde{f}_{i}(0)-n)(4\pi)^{-\frac{n}{2}}d\widetilde{g}_0^i\\
 & \geq\limsup_{i\to\infty}\mathcal{W}(\widetilde{g}_0^i,\widetilde{f}_{i}(0),1).
\end{align*}
Taking the infimum over all compact subsets $K\subseteq\mathcal{R}$
gives $\mathcal{W}(g_{\infty},f_{\infty})\geq\Theta(q)$. Now fix
$\epsilon>0$, and choose $K\subseteq\mathcal{R}$ compact such that
\[
(4\pi)^{-\frac{n}{2}}\int_{\mathcal{R}\setminus K}|R_{g_{\infty}}+|\nabla f_{\infty}|^{2}+f_{\infty}-n|e^{-f_{\infty}}dg_{\infty}<\epsilon.
\]
Then, for any $K'\subseteq\mathcal{R}$ compact with $K\subseteq K'$,
we have 
\begin{align*}
\mathcal{W}(g_{\infty},f_{\infty}) & \leq\int_{K'}(R_{g_{\infty}}+|\nabla f_{\infty}|^{2}+f_{\infty}-n)(4\pi)^{-\frac{n}{2}}e^{-f_{\infty}}dg_{\infty}+\epsilon\\
 & =\lim_{i\to\infty}\int_{\Phi_{i}(K')}(R_{\widetilde{g}_0^i}+|\nabla\widetilde{f}_{i}(0)|^{2}+\widetilde{f}_{i}(0)-n)(4\pi)^{-\frac{n}{2}}e^{-\widetilde{f}_{i}(0)}d\widetilde{g}_0^i+\epsilon.
\end{align*}
In order to show $\mathcal{W}(g_{\infty},f_{\infty})\leq\Theta(q)$,
it therefore suffices to find some $K'\subseteq\mathcal{R}$ compact
(possibly depending on $\epsilon$) with $K\subseteq K'$ and 
\[
\liminf_{i\to\infty}\int_{M\setminus\Phi_{i}(K')}(R_{\widetilde{g}_0^i}+|\nabla\widetilde{f}_{i}(0)|^{2}+\widetilde{f}_{i}(0)-n)(4\pi)^{-\frac{n}{2}}e^{-\widetilde{f}_{i}(0)}d\widetilde{g}_0^i>-\epsilon.
\]
Since $\widetilde{f}_{i}(0)$ have uniform quadratic growth, and because
$|R_{\widetilde{g}_0^i}|\leq A$, we can find $D=D(A,n)<\infty$
uniform such that 
\[
R_{\widetilde{g}_0^i}+|\nabla\widetilde{f}_{i}(0)|^{2}+\widetilde{f}_{i}(0)-n\geq0\hspace{6mm}\mbox{ on }\hspace{6mm}M\setminus B_{\widetilde{g}^i}(q,0,D)
\]
for all $i\in\mathbb{N}$. Moreover, Bamler's upper bound (\thref{bamthm}) on the size of the quantitative singular set gives us $E=E(A,n)<\infty$
such that 
\[
|\{r_{Rm}^{\widetilde{g}^i}(\cdot,0)>s\}\cap B_{\widetilde{g}_0^i}(q,2D)|_{\tilde{g}_0^i}\leq Es^{3}
\]
for all $s\in(0,1]$.

We also know that the entropy integrand is bounded uniformly from
below on $B_{\widetilde{g}^i}(q,0,D)$, and that for $i\in\mathbb{N}$
sufficiently large we have 
\[
\{r_{Rm}^{\widetilde{g}^i}(\cdot,0)\geq s\}\cap B_{\widetilde{g}^i}(q,0,2D)\subseteq V_{i}.
\]
Thus we can choose $s=s(A,n,\epsilon)>0$ sufficiently small so that
\[
\int_{\{r_{Rm}^{\widetilde{g}^i}(\cdot,0)<s\}\cap B_{\widetilde{g}^i}(q,0,2D)}(R_{\widetilde{g}_0^i}+|\nabla\widetilde{f}_{i}(0)|^{2}+\widetilde{f}_{i}(0)-n)(4\pi)^{-\frac{n}{2}}e^{-\widetilde{f}_{i}(0)}d\mu_{\widetilde{g}_0^i}>-\epsilon.
\]
Finally, by the definition of a convergence scheme, we can choose
$K'\subseteq\mathcal{R}$ such that $K\subseteq K'$ and $\Phi_{i}(K)\supseteq\{r_{Rm}^{\widetilde{g}^i}(\cdot,0) \geq s\}\cap B_{\widetilde{g}^i}(q,0,2D)$
(in fact, this will follow by taking $K'=\overline{U}_{i}$ for some large $i\in\mathbb{N}$).
\end{proof}
\begin{defn} A singular GRS $(\mathcal{R},g,f)$ is called normalized if $$\int_{\mathcal{R}}(4\pi)^{-\frac{n}{2}}e^{-f}dg=1$$.\end{defn}
Recall that $R+|\nabla f|^{2}-f$ is some constant $c\in\mathbb{R}$ and $R+\Delta f=\frac{n}{2}$, we can write
\begin{align*}
\mathcal{W}(g,f)= & (4\pi)^{-\frac{n}{2}}\int_{\mathcal{R}}(R+2\Delta f-|\nabla f|^{2}+f-n)e^{-f}dg\\
= & (4\pi)^{-\frac{n}{2}}\int_{\mathcal{R}}(-R-|\nabla f|^{2}+f)e^{-f}dg=-c\int_{\mathcal{R}}(4\pi)^{-\frac{n}{2}}e^{-f}dg.
\end{align*}
That is, for a normalized soliton, we know $R+|\nabla f|^{2}=f-\mathcal{W}(g,f)$.
\begin{prop} \label{normy}
The singular shrinking GRS $(\mathcal{R},g_{\infty},f_{\infty})$ of \thref{entcon} is normalized: 
\[
\int_{\mathcal{R}}(4\pi)^{-\frac{n}{2}}e^{-f_{\infty}}dg_{\infty}=1.
\]
\end{prop}

\begin{proof}
For any compact subset $K\subseteq\mathcal{R}$, we have 
\[
\int_{K}e^{-f_{\infty}}dg_{\infty}=\lim_{i\to\infty}\int_{\varphi_{i}(K)}e^{-\widetilde{f}_{i}(0)}d\widetilde{g}_0^i \leq(4\pi)^{\frac{n}{2}},
\]
so it suffices to prove that $\int_{\mathcal{R}}(4\pi)^{-\frac{n}{2}}e^{-f_{\infty}}dg_{\infty}\geq1$.
In fact, fix $\epsilon>0$. By the uniform volume upper bound (Proposition \ref{volumebounds}) and
heat kernel lower bound (Lemma \ref{heatsing}), we have some $D=D(\epsilon)<\infty$ such that
\begin{align*}
\int_{M\setminus B_{\widetilde{g}^i}(q,0,D)}(4\pi)^{-\frac{n}{2}}e^{-\widetilde{f}_{i}(0)}d\widetilde{g}_0^i\leq & C(n,A)\int_{M\setminus B_{\widetilde{g}^i}(q,0,D)}\exp\left(-\frac{1}{C}\widetilde{d}_{g_0^i}^{2}(q,x)\right)d\widetilde{g}_0^i\\
    = & C(n,A)\int_{D}^{\infty}\text{Area}_{\widetilde{g}^i}(\partial B_{\widetilde{g}^i}(q,0,r))e^{-\frac{r^{2}}{C}}dr\\
= & C(n,A)\int_{D}^{\infty}e^{-\frac{r^{2}}{C}}\dfrac{d}{dr}|B_{\widetilde{g}^i}(q,0,r)|_{\widetilde{g}_0^i}dr\\
\leq & C(n,A)\int_{D}^{\infty}r|B_{\widetilde{g}^i}(q,0,r)|_{\widetilde{g}_0^i}e^{-\frac{r^{2}}{C}}dr\\
\leq & C(n,A)\int_{D}^{\infty}r^{n+1}e^{-\frac{r^{2}}{C}}dr <\frac{\epsilon}{2}
\end{align*}
for all $i\in \mathbb{N}$. Moreover, since $e^{-f_{\infty}}$
is uniformly bounded on $B_{\widetilde{g}^i}(q,0,2D)$ (independently
of $i$) we also have 
\[
\int_{B_{\widetilde{g}^i}(q,0,2D)\setminus V_{i}}(4\pi)^{-\frac{n}{2}}e^{-\widetilde{f}_{i}(0)}d\mu_{\widetilde{g}_0^i}\leq C\left|\{r_{Rm}^{\widetilde{g}^i}(\cdot,0)>s\}\cap B_{\widetilde{g}^{i}}(q,0,2D)\right|_{\widetilde{g}_0^i}
\]
for any $s>0$, when sufficiently large $i$. By taking $s>0$ sufficiently
small, the upper bound on the size of the quantitative singular set
(as in the previous section) tells us that the right hand side is
less than $\frac{1}{2}\epsilon$. This means that 
\[
\int_{B_{\widetilde{g}^i}(q,0,2D)\cap V_{i}}(4\pi)^{-\frac{n}{2}}e^{-\widetilde{f}_{i}(0)}d\mu_{\widetilde{g}_0^i}\geq1-\epsilon
\]
for $i$ sufficiently large, hence 
\[
\int_{\mathcal{R}\cap B(q,2D)}(4\pi)^{-\frac{n}{2}}e^{-f_{\infty}}d\mu_{g_{\infty}}\geq1-2\epsilon,
\]
and the claim follows. $\square$
\end{proof}
\begin{rem} \thlabel{remarque}
As in the Type-I curvature case \cite{mant}, we note that Proposition \ref{normy} and the entropy convergence part of Theorem \ref{entcon} also hold if the sequence $\widetilde{f}_{i}$ is replaced by $f(\cdot,t_i+|t_i|t)$ for some fixed $f\in \mathcal{F}_q$. The equality $\mathcal{W}(g_{\infty},f_{\infty})=\Theta(q)$ could fail a priori in that setting, though equality will follow from the results of Section 6.
\end{rem}

\section{Entropy Rigidity of the Gaussian Soliton}

\vspace{4mm}

\noindent The following result extends Lemma 2.1 of \cite{naber} to
the setting of singular shrinking GRS. The proof of that lemma used
essentially the fact that the underlying Riemannian manifold is complete,
which in our setting is only true if the singular set $X\setminus\mathcal{R}$
is empty. However, we will see that the proof can be modified to work
when $X\setminus\mathcal{R}$ has singularities of codimension strictly
greater than 3, using the arguments of Claim 2.32 of \cite{bingspace2}.
In fact, the part of the following proof establishing the flow properties
of a function $f\in C^{\infty}(\mathcal{R})$ with $\nabla^{2}f=0$
is taken from this claim, but since the setting of \cite{bingspace2}
is somewhat different, we rewrite the part of this claim we need. 
\begin{prop} \label{naberextend}
Suppose $\mathcal{X}=(X,d,\mathcal{R},g,f_{i})$, $i=1,2$ are normalized
singular shrinking GRS with singularities of codimension 4. Then 
\[
\mathcal{W}(g,f_{1})=\mathcal{W}(g,f_{2}).
\]
\end{prop}

\begin{proof}
We can assume that $f_{1}-f_{2}$ is not constant, otherwise the normalization
condition gives the claim. Set $f:=|\nabla(f_{1}-f_{2})|^{-1}(f_{1}-f_{2})$,
so that $|\nabla f|=1$ and $\nabla^{2}f=0$ on $\mathcal{R}$. Let
$\varphi_{t}(x)$ be the flow of $\nabla f$ starting at $x\in\mathcal{R}$
for $t\in\mathbb{R}$ such that this is defined. Fix $p\in(2,4)$,
$s\in(0,1]$. We first show that, for any $q\in X$, $s\in(0,1]$,
and $D<\infty$, the set 
\[
S_{D,s}:=\{x\in\mathcal{R}\cap B^{X}(q,D)\:;\:r_{Rm}^{X}(\varphi_{t}(x))<\frac{1}{2}s\mbox{ for some }t\in[-D,D]\}
\]
has Minkowski codimension at least $p-1$. We denote by $\mathcal{H}^{n-1}$
the $(n-1)$-dimensional Hausdorff measure on $\mathcal{R}$, which
coincides with the Lebesgue measure on any hypersurface.

Because $r_{Rm}^{X}$ is 1-Lipschitz, we can find $h\in C^{\infty}(\mathcal{R})$
such that $|\nabla h|\leq2$ and 
\[
\dfrac{1}{2}r_{Rm}^{X}<h<2r_{Rm}^{X}\hspace{6mm}\mbox{ on }\mathcal{R}.
\]
Using the coarea formula and the fact that singularities are of codimension
4, we have 
\begin{align*}
\int_{s}^{2s}\mathcal{H}^{n-1}(h^{-1}(t)\cap B^{X}(q,3D))dt= & \int_{\{s\leq h\leq2s\}\cap B^{X}(q,3D)}|\nabla h|dg\\
\leq & 2|\{r_{Rm}^{X}\leq4s\}\cap B^{X}(q,3D)\cap\mathcal{R}|\leq2\cdot4^{p}E_{p,3D,q}s^{p}.
\end{align*}
By Sard's theorem, we may therefore find $t=t(s)\in(s,2s)$ such that
$\Sigma_{s}:=h^{-1}(t)\cap B^{X}(q,3D)$ is smooth and satisfies $\mathcal{H}^{n-1}(\Sigma_{s})\leq2^{2p+1}E_{p,3D,q}s^{p-1}$.
Next, write $S_{D,s}=I_{s}\cup II_{s}$, where 
\[
I_{s}:=\{x\in S_{D,s}\:;\:r_{Rm}^{X}(x)\leq4s\},
\]
\[
II_{s}:=\{x\in S_{D,s}\:;\:r_{Rm}^{X}(x)>4s\}.
\]
Since the singularities of $\mathcal{X}$ are codimension 4, we have
\[
|I_{s}|\leq|\{r_{Rm}^{X}\leq4s\}\cap B^{X}(q,D)\cap\mathcal{R}|\leq4^{p}E_{p,3D,q}s^{p}.
\]
For any $y\in II_{s}$, there exists $t\in(-D,D)$ such that $\varphi_{t}(y)\in\Sigma_{s}$.
Moreover, $|\nabla f|=1$ implies $d(\varphi_{t}(y),q)\leq2D<3D$.
Now set 
\[
\Omega_{s}:=\{(t,x)\in(-D,D)\times\Sigma_{s}\:;\:\varphi_{t}(x)\mbox{ is well defined}\},
\]
which is open in $(-D,D)\times\Sigma_{s}$.\\
 
\noindent \textbf{Claim 1:} The Jacobian of $\eta:(\Omega_{s},dt^{2}+g_{\Sigma_{s}})\to(\mathcal{R},g),(t,x)\mapsto\varphi_{t}(x)$
is $\leq1$ everywhere.\\
 In fact, since each $\varphi_{t}$ is a local isometry, we have that
$$d\eta_{(t,x)}|_{T_{x}\Sigma_{s}}:T_{x}\Sigma_{s}\to T_{\varphi_{t}(x)}(\varphi_{t}(\Sigma_{s}))$$
is a linear isometric embedding for all $(t,x)\in(-D,D)\times\Sigma_{s}$.
Moreover, $d\eta_{(t,x)}(\partial/\partial t)=\nabla f(\varphi_{t}(x))$,
and so the Jacobian of $\eta$ at $(t,x)\in\Sigma_{s}\times(-D,D)$
is $|(\nabla f(\varphi_{t}(x))^{\perp}|\leq1$, where $\perp$ denotes
the projection $T_{\varphi_{t}(x)}\mathcal{R}\to(T_{\varphi_{t}(x)}\varphi_{t}(\Sigma_{s}))^{\perp}$.
$\square$\\
Note that $II_{s}\subseteq\eta(\Omega_{s})$ and the claim give $|\eta(\Omega_{s})|\leq\mathcal{H}^{n-1}(\Sigma_{s})\cdot2D\leq4^{p+2}E_{q,3D,p}Ds^{p-1}$,
so we may conclude 
\[
|S_{D,s}|=|I_{s}|+|II_{s}|\leq4^{p+3}E_{q,3D,p}s^{p-1}.
\]\\
\noindent \textbf{Claim 2:} $\{x\in\mathcal{R}\:;\:d(x,S_{D,s})<s\}\subseteq S_{2D,4s}$.\\
 In fact, suppose $x\in\mathcal{R}$ satisfies $d(x,S_{D,s})<s$.
If $r_{Rm}^{X}(x)<2s$, then $x\in S_{2D,4s}$ by definition. If $r_{Rm}^{X}(x)>2s$,
then there is a minimal geodesic from $x$ to some point in $S_{2D,s}$,
and this geodesic lies entirely in $\{r_{Rm}^{X}>s\}\subseteq\mathcal{R}$.
By construction, there is some $t\in(-D,D)$ such that $\varphi_{t}(\gamma)\cap\{r_{Rm}^{X}\leq\frac{1}{2}\}\neq\emptyset.$
Let $t_{0}\in(-D,D)$ be such that $|t_{0}|$ is minimal among such
$t$. We can assume, by replacing $f$ with $-f$, that $t_{0}>0$.
Then, since $\varphi_{t_{0}}$ is a local isometry, and $r_{Rm}^{X}\geq\frac{1}{2}s$
along $\varphi_{t}(\gamma)$ for $0\leq t\leq t_{0}$, we know that
$\varphi_{t_{0}}$ is defined on $\gamma$ and that $L_{g}(\varphi_{t_{0}}(\gamma))=L_{g}(\gamma)<s$.
Also, by construction we have $\varphi_{t_{0}}(\gamma)\cap\{r_{Rm}^{X}=\frac{1}{2}s\}\neq\emptyset$.
Since $r_{Rm}^{X}$ is 1-Lipschitz, this implies $r_{Rm}^{X}(\varphi_{t_{0}}(x))\leq3/2$,
hence $x\in S_{2D,4s}$.\\

This along with $|S_{2D,2s}|\leq4^{p+10}E_{q,6D,p}s^{p-1}$ implies
the Minkowski dimension claim. In particular, the set $S$ of $x\in\mathcal{R}$
such that $\varphi_{t}(x)$ does not exist for all time satisfies
$|S|=0$ and $\mathcal{H}^{n-1}(S\cap f^{-1}(0))=0$. Define $N:=f^{-1}(0)\cap\mathcal{R}$,
and let $U\subseteq\mathbb{R}\times N$ be the (open) maximal subset
where $\psi(t,x):=\varphi_{t}(x)$ is defined. Then $\mathcal{R}\setminus S\subseteq\psi(U)$,
since for any $x\in\mathcal{R}\setminus S$, we have $x=\psi(f(x),\varphi_{-f(x)}(x))$.
In particular, $|\mathcal{R}\setminus\psi(U)|=0$. By an computation similar to that in Claim 1, and noting that now $(\nabla f(\varphi_{t}(x)))^{\perp}=\nabla f(\varphi_{t}(x))$,
where $\perp$ denotes the projection $T_{\varphi_{t}(x)}\mathcal{R}\to(T_{\varphi_{t}(x)}\varphi_{t}(N))^{\perp}$,
we get that $\psi(U)$ is a Riemannian isometry $(U,dt^{2}+\widetilde{g})\to(\psi(U),g)$,
where $\widetilde{g}$ is the Riemannian metric $\widetilde{g}$ on
$N:=f^{-1}(0)\cap\mathcal{R}$ induced from $g$. In particular, $\widetilde{f}_{i}:=f_{i}\circ\psi\in C^{\infty}(U)$
are soliton functions, and $(f\circ\psi)(t,x)=t$.\\

\noindent \textbf{Claim 3:} There are $a_{i}\in\mathbb{R}$ such that 
\[
\widetilde{f}_{i}(t,x)=\widetilde{f}_{i}(0,x)+a_{i}t+\frac{1}{4}t^{2}.
\]
for all $(t,x)\in U$.

In fact, the pulled back soliton equation gives $\partial_{t}^{2}\widetilde{f}_{i}=\frac{1}{2}$
everywhere, so 
\[
\widetilde{f}_{i}(t,x)=\widetilde{f}_{i}(0,x)+\partial_{t}\widetilde{f}_{i}(0,x)t+\frac{1}{4}t^{2}
\]
for $(t,x)\in U$. Moreover, for any $X\in\mathfrak{X}(N)$,  we have $\nabla_X \partial_t=0$ ,  so the Riemannian product structure and the soliton equation give 
\[
X(\partial_{t}\widetilde{f}_{i})=\nabla^{2}\widetilde{f}_{i}(\partial_{t},X)=\frac{1}{2}g(\partial_{t},X)-Rc(\partial_{t},X)=0.
\]
This means that $\nabla(\partial_{t}\widetilde{f}_{i}-\frac{1}{2}t)=0$
on $U$, hence $\nabla(\langle\nabla f,\nabla f_{i}\rangle-\frac{1}{2}f)=0$
on the dense open subset $\psi(U)$ of $\mathcal{R}$. Because $f,f_{i}$
are smooth and $\mathcal{R}$ is connected, we get that $\langle\nabla f,\nabla f_{i}\rangle-\frac{1}{2}f$
is constant on $\psi(U)$, hence $\partial_{t}\widetilde{f}_{i}-\frac{1}{2}t$
is constant on $U$. In particular, $\partial_{t}\widetilde{f}_{i}$
is constant on $\{0\}\times N$, and the claim follows. $\square$ \\

Now we use the normalization conditions on $\widetilde{f}_{i}$. Since
$|\mathcal{R}\setminus\psi(U)|=0$ and $\partial_{t}(\widetilde{f}_{1}-\widetilde{f}_{2})=1$,
we have 
\begin{align*}
(4\pi)^{-\frac{n}{2}}\left(\int_{N}e^{-\widetilde{f}_{2}(0,x)}d\widetilde{g}(x)\right)\left(\int_{\mathbb{R}}e^{-\frac{1}{4}t^{2}-a_{2}t}dt\right) & =1,\\
(4\pi)^{-\frac{n}{2}}\left(\int_{N}e^{-\widetilde{f}_{2}(0,x)}d\widetilde{g}(x)\right)\left(\int_{\mathbb{R}}e^{-\frac{1}{4}t^{2}-a_{1}t}dt\right) & =1,
\end{align*}
since $\widetilde{f}_{1}=\widetilde{f}_{2}$ on $\{0\}\times N$.
Thus 
\[
e^{a_{2}^{2}}\int_{\mathbb{R}}e^{-\frac{1}{4}t^{2}}dt=e^{a_{2}^{2}}\int_{\mathbb{R}}e^{-\frac{1}{4}(t-2a_{2})^{2}}dt=e^{a_{1}^{2}}\int_{\mathbb{R}}e^{-\frac{1}{4}(t-2a_{1})^{2}}dt=e^{a_{1}^{2}}\int_{\mathbb{R}}e^{-\frac{1}{4}t^{2}}dt,
\]
which implies $a_{1}^{2}=a_{2}^{2}$. Noting that $\nabla_{X}\widetilde{f}_{1}(0,x)=\nabla_{X}\widetilde{f}_{2}(0,x)$
for all $x\in N$ and $X\in T_x N$, we thus have 
\[
|\nabla\widetilde{f}_{1}(0,x)|^{2}-|\nabla \widetilde{f}_2(0,x)|^2 = a_1^2 -a_2^2 =0.
\]
In particular, on $\{0\}\times N$, 
\[
R+|\nabla\widetilde{f}_{1}|^{2}-\widetilde{f}_{1}=R+|\nabla\widetilde{f}_{2}|^{2}-\widetilde{f}_{2}.
\]
Since $f_{i}$ are normalized, we have $R+|\nabla f_{i}|^{2}-f_{i}=-\mathcal{W}(g,f_{i})$,
so the proposition follows.
\end{proof}
Next, we address the rigidity statement of Theorem \ref{mainthm}.

\begin{prop} \thlabel{rigid}
Suppose $(M,(g_t)_{t\in[-2,0)},q)$ is a closed, pointed solution of
Ricci flow with $$\sup_{t\in[-2,0)}|R(\cdot,t)|(T-t)<\infty,$$ and
let $(\mathcal{X},q_{\infty})$ be a singular shrinking GRS obtained
as a Type-I limit. If $\mathcal{W}(g_{\infty},f_{\infty})=0$, then
$\mathcal{X}$ is the Gaussian shrinker. If this occurs, there is
a neighborhood $U$ of $q$ in $M$ such that $$\sup_{U\times[-2,0)}|Rm|<\infty.$$
\end{prop}

\begin{proof}
Fix $x\in\mathcal{R}$, and let $(U_{i},V_{i},\Phi_{i})$ be the convergence
scheme for $(M,\widetilde{g}_{0}^{i},q)\to(\mathcal{X},q_{\infty})$.
Note that 
\[
d^{X}(x,q_{\infty})=\lim_{i\to\infty}d_{\widetilde{g}_{0}^{i}}(\Phi_{i}(x),q)=\lim_{i\to\infty}|t_{i}|^{-\frac{1}{2}}d_{g_{t_{i}}}(\phi_{i}(x),q),
\]
so $\phi_{i}(x)\to q$ in $M$. Thus, after passing to a subsequence, $u_{\Phi_{i}(x),t_{i}}$ converges
to a conjugate heat kernel at the singular time $u\in\mathcal{U}_{q}$ in $C_{loc}^{\infty}(M\times(-1,0))$.
Writing $u(y,s)=(4\pi|s|)^{-\frac{n}{2}}e^{-f(y,s)}$, we know from
previous sections that, if $f_{i}(s):=f(t_{i}+|t_{i}|s)$, then $f_{i}(0)\circ\Phi_{i}$
converges in $C_{loc}^{\infty}(\mathcal{R})$ to a normalized soliton
function $\overline{f}_{\infty}$, which must satisfy $\mathcal{W}(g_{\infty},\overline{f}_{\infty})=\mathcal{W}(g_{\infty},f_{\infty})=0$
by Remark \ref{remarque} and Proposition \ref{naberextend}, hence (again using Remark \ref{remarque})
\begin{align*}
0&=\lim_{i\to\infty}\mathcal{W}(g_{0}^{i},f_{i}(0),1)=\lim_{i\to\infty}\mathcal{W}(|t_{i}|^{-1}g_{t_{i}},f(t_{i}),1)\\&=\lim_{i\to\infty}\mathcal{W}(g_{t_{i}},f(t_{i}),|t_{i}|)=\lim_{t\nearrow0}\mathcal{W}(g_{t},f(t),|t|)
\end{align*}
by \thref{entcon}. Now let $\epsilon=\epsilon(n,C)>0$ be the constant from Theorem \ref{epsilonreg}.
Then there exists $\delta>0$ such that $\mathcal{W}(g_{t},f(t),|t|)\geq-\frac{1}{2}\epsilon$
for all $t\in[-\delta,0)$. Because $f_{\Phi_{i}(x),t_{i}}\to f$
in $C_{loc}^{\infty}(M\times(-1,0))$, we know that for any fixed
$t\in(-1,0)$, we have
\[
\mathcal{W}(g_{t},f(t),|t|)=\lim_{i\to\infty}\mathcal{W}(g_{t+t_{i}},f_{\Phi_{i}(x),t_{i}}(t+t_{i}),|t|).
\]
In particular, $\mathcal{W}(g_{-\delta+t_{i}},f_{\Phi_i(x),t_{i}}(-\delta+t_{i}),\delta)\geq-\epsilon$
for sufficiently large $i\in\mathbb{N}$. By Theorem \ref{epsilonreg}, we conclude
$\left(r_{Rm}^{g}(\Phi_{i}(x),t_{i})\right)^{2}\geq\epsilon\delta$.
This means $\left(r_{Rm}^{\widetilde{g}^{i}}(\Phi_{i}(x),0)\right)^{2}>\epsilon\delta|t_{i}|^{-1}$,
so by backwards Pseudolocality, it follows that $(M,|t_{i}|^{-1}g_{t_{i}},q)$
actually converges in the $C^{\infty}$ Cheeger-Gromov sense to the
Gaussian shrinker on flat $\mathbb{R}^{n}$. 

Now apply a version of Perelman's
pseudolocality theorem (Theorem 1.2 of \cite{penglu}) to the ball $B(q,t_{i},D\sqrt{|t_{i}|})$,
with $D<\infty$ and $i\in \mathbb{N}$ sufficiently large, to conclude that 
$|Rm|(x,t)\leq C$ for all $x\in B(q,t_{i},\sqrt{|t_{i}|})$,
$t\in(t_{i},0)$, (see also Lemma 2.4 of \cite{topping}). 
\end{proof}
\noindent \emph{Proof of Theorem 1.} By Section 3, we can pass to a further subsequence in order to assume that $\widetilde{f}_i(0)$ converge to another smooth soliton potential function $f_{\infty}'\in C^{\infty}(\mathcal{R})$, which satisfies $\mathcal{W}(g_{\infty},f_{\infty}')=\Theta(q)$.  By Proposition \ref{naberextend}, we have $\mathcal{W}(g_{\infty},f_{\infty}')=\mathcal{W}(g_{\infty},f_{\infty})$. The remaining claim is Proposition \ref{rigid}. $\square$

\section{Removable Singularities}

In this section, we specialize to the four-dimensional case, where we first
sharpen Bamler's Minkowski dimension estimates for the singular
set, obtaining that the limiting singular GRS is actually smooth outside of a
discrete set of points. Using this, we are able to show the singularities
are conical $C^{0}$ orbifold singularities, without knowing that
the global $L^{2}$ norm of the curvature tensor on the regular set
is finite (this is not true in general, even if we assume
\eqref{eq:curv} so that $\mathcal{X}$ is smooth). In fact, it is not clear how one can prove local $L^2$ estimates for the curvature on the rescaled Ricci flow. This is because the $L^2$ curvature bound in dimension 4 is usually proved using the Chern-Gauss-Bonnet formula, but the argument relies crucially on the (rescaled) flow having uniformly bounded diameter. Moreover, it is not clear how to effectively localize the Chern-Gauss-Bonnet formula in this situation: applying the formula on a subdomain results in boundary terms which depend on the principal curvatures of the boundary. In \cite{hasl}, this difficulty was overcome by using properties of level sets of a shrinking GRS, which suggests that it may be easier to prove the $L^2$ curvature estimate on the limiting singular space rather than on the Ricci flow itself.

Therefore, we aim to prove
a local $L^{2}$ bound for $|Rm|$ near the singular points of $\mathcal{X}$, and then apply the removable singularity techniques of \cite{tianchern},\cite{caosesum}, \cite{uhlenbeck}. We achieve this
by estimating separately the traceless Ricci and the Weyl parts of
the curvature tensor, using ideas of Haslhoffer-Muller \cite{hasl} and Donaldson-Sun \cite{donaldsun},
respectively. After overcoming this difficulty, the proof is fairly
standard, and Uhlenbeck's theory \cite{uhlenbeck} of removable singularities along
with the $\epsilon$-regularity theorem proved in \cite{shaosai}, and later \cite{jiangreg}, let us conclude
that in fact the singular GRS has a $C^{\infty}$ orbifold structure.

Throughout this section, we suppose that $(M^{4},(g_{t})_{t\in[-2,0)})$
is a closed solution of Ricci Flow satisfying $\nu[g_{-2},4]\geq -A$ and
\[
|R(x,t)|\leq \dfrac{A}{|t|}
\]
for all $(x,t)\in M\times[-2,0)$. Fix a basepoint $q\in M$ and a
sequence of times $t_{i}\nearrow0$. Define the rescaled sequence
$\widetilde{g}_{t}^{i}:=|t_{i}|^{-1}g_{t_{i}+|t_{i}|t}$ for $t\in[-2,0)$. Then
the rescaled solutions satisfy $\sup_{M\times [-2,0]}|R_{\widetilde{g}^i}|\leq A$ and $\nu[\widetilde{g}_{-2}^i,4]\geq -A$
for all $i\in\mathbb{N}$. By Theorem 1.2 of \cite{bam2}, we may pass to a subsequence so that $(M,\widetilde{g}_{0}^{i},q)$ converges
to a pointed singular space $(\mathcal{X},q_{\infty})=(X,d,\mathcal{R},g,q_{\infty})$
with singularities of codimension 4, that is $Y$-regular at all scales,
for some $Y=Y(A)<\infty$, and satisfies the shrinking soliton equation
$Rc+\nabla^{2}f=\frac{1}{2}g$ on the regular part $\mathcal{R}$,
where $f\in C^{\infty}(\mathcal{R})$ is the obtained from a sequence
of rescaled conjugate heat kernels based at the singular time. We recall that $|R|\leq A$
on $\mathcal{R}$, and that $f$ satisfies quadratic growth estimates (\ref{fboundslimit}), which combine with the equation $R+|\nabla f|^{2}=f-\mathcal{W}(g,f)$ to
give a locally uniform gradient estimate for $f$. 
\begin{lem} \label{tech}
$X\setminus\mathcal{R}$ is discrete, and every tangent cone at $x\in X\setminus\mathcal{R}$
is isometric to $\mathbb{R}^{4}/\Gamma$ for some finite subgroup
$\Gamma\leq O(4,\mathbb{R})$ (which may depend on $x$ and the choice
of rescalings). Moreover, there exists $N=N(A)>0$ such that $|\Gamma|\leq N$.
\end{lem}

\begin{proof}
Fix $x_{0}\in X\setminus\mathcal{R}$, and let $(Z,d_{Z},c_{Y})$
be a tangent cone at $x_{0}$, with $\lambda_{i}\to\infty$ such that
$(X,\lambda_{i}d_{X},x_{0})\to(Z,d_{Z},c_{Z})$ in the pointed Gromov-Hausdorff
sense. By Corollary 1.5 of \cite{bam2}, $Z$ is a metric cone. Choose $x_{i}\in M$ such that $x_{i}\to x_{0}$ as $i\to\infty$.
By definition of the convergence $(M,g_{0}^{i},q)\to(\mathcal{X},q_{\infty})$,
for each $i\in\mathbb{N}$, we can choose $j=j(i)\geq i$ such that
$(M,\lambda_{i}^{2}g_{0}^{j(i)},x_{j(i)})$ is $\lambda_{i}i^{-1}$-close
in the pointed Gromov-Hausdorff topology to $(X,\lambda_{i}d_{X},x_{0})$.
Setting $\widetilde{g}_{t}^{i}:=\lambda_{i}^{2}g_{\lambda_{i}^{-2}t}^{j(i)}$,
we get that $(M,(\widetilde{g}_{t}^{i})_{t\in[-2,0]},x_{j(i)})_{i\in\mathbb{N}}$
is a sequence of pointed Ricci flows with $\sup_{M\times[-2,0]}|R_{\widetilde{g}^{i}}|\to0$
and $\nu[\widetilde{g}_{-2}^{i},4]\geq-A$, which converges in the
pointed Gromov-Hausdorff sense to $(Z,d_{Z},c_{Z})$. In particular,
$(Z,d_{Z},c_{Z})$ has the structure of a singular space $\mathcal{Z}=(Z,d_{Z},\mathcal{R}_{Z},g_{Z},c_{Z})$
with mild singularities of codimension 4, such that $Rc_{g_Z}=0$
on $\mathcal{R}_{Z}$. However, $Z=C(\Sigma)$ is
a metric cone, so the link $\Sigma$ of $Z$
is a smooth 3-dimensional Riemannian manifold. That is, $Z\setminus\{c_{Z}\}$
is a smooth metric cone $g_{Z}=dr^{2}+r^{2}g_{\Sigma}$ for some smooth
Riemannain metric $g_{\Sigma}$ on $\Sigma$. However, $Rc_{g_Z}=0$
implies $Rc_{g_{\Sigma}}=(n-1)g_{\Sigma}$, and since $\dim(\Sigma)=3$,
$(\Sigma,g_{\Sigma})$ must be a disjoint union of spherical space
forms. Because $\mathcal{R}_{Z}=Z \setminus \{ c_Z\}$ is connected, $\Sigma$ must
be connected. Thus,  $Z=C(\mathbb{S}^{3}/\Gamma)=\mathbb{R}^{4}/\Gamma$
for some finite subgroup $\Gamma\leq O(4,\mathbb{R})$. Moreover,
because $Z$ is $Y$-tame for some $Y=Y(A)<\infty$ (by Proposition
4.2 of \cite{bam1}), we have
\[
c(A)<|B^{Z}(c_{Z},1)\setminus\{c_{Z}\}|_{g_Z}=\omega_{n}/|\Gamma|.
\]

It remains to show that $x_{0}$ is an isolated point of $X\setminus\mathcal{R}$. Suppose by way of contradiction that there exist $y_{i}\in X\setminus(\mathcal{R}\cup\{x_{0}\})$
such that $y_{i}\to x_{0}$. Set $\lambda_{i}:=1/d(x_{0},y_{i})$.
By passing to a subsequence, we can assume $(X,\lambda_{i}d_{X},x_{0})$
converges in the pointed Gromov-Hausdorff sense to a tangent cone
$(Z,d_{Z},c_{Z})$ as above. For any $\alpha\in(0,1)$, we can pass to a further subsequence so that $(B^{X}(y_{i},\alpha\lambda_{i}^{-1}),\lambda_{i}d,y_{i})$ converges
in the pointed Gromov-Hausdorff sense to $(B^{Z}(y_{\infty},\alpha),d_{Z},y_{\infty})$
for some $y_{\infty}\in Z$ with $d(c_{Z},y_{\infty})=1$. 
By possibly shrinking $\alpha>0$, we can assume that $B^Z(y_{\infty},\alpha)$ is isometric to a ball in $\mathbb{R}^n$. Applying Theorem 2.37 of \cite{tianfano} (see the appendix of this paper), we have
 $$|B^{X}(y_{i},\alpha\lambda_{i}^{-1})\cap\mathcal{R}|_g \geq(\omega_{n}-\epsilon_{i})(\alpha\lambda_{i}^{-1})^4$$
for some sequence $\epsilon_{i}\to0$. However, the $Y(A)$-regularity
of $\mathcal{X}$ then implies $r_{Rm}(y_{i})>0$, contradicting $y_{i}\in X\setminus\mathcal{R}$. 
\end{proof}
\begin{thm} \thlabel{removable}
$\mathcal{X}$ has the structure of a $C^{\infty}$ Riemannian orbifold
with finitely many conical orbifold singularities, such that in orbifold charts
around the singular points, $f$ extends smoothly across the singular
points, and satisfies the gradient Ricci soliton equation everywhere. 
\end{thm}

\begin{proof}
Fix $x^{\ast}\in X\setminus\mathcal{R}$. Suppose by way of contradiction
that there exists a sequence $x_{i}\to x^{\ast}$ such that $\liminf_{i\to\infty}|Rm(x_{i})|d_{X}^{2}(x_{i},x^{\ast})>0$
(since $x^{\ast}$ is an isolated point of $X\setminus\mathcal{R},$we
can assume $x_{i}\in\mathcal{R}$). Set $r_{i}:=d_{X}(x_{i},x^{\ast})$,
so that by passing to a subsequence, we may assume that $(X,r_{i}^{-1}d_{X},x^{\ast})$
converge in the pointed Gromov-Hausdorff sense to $(C(\mathbb{S}^{3}/\Gamma),d_{C(\mathbb{S}^{3}/\Gamma)},c_{0})$
for some finite subgroup $\Gamma\leq O(4,\mathbb{R})$, where $c_{0}\in C(\mathbb{S}^{3}/\Gamma)$
is the cone point. 

We claim that there are $\epsilon_{i}\to0$ such that, for all $x\in B^{X}(x^{\ast},2r_{i})\setminus\overline{B}^{X}(x^{\ast},\frac{1}{2}r_{i})$,
we have the Gromov-Hausdorff distance estimate 
$$d_{GH} \left( (B^{X}(x,\alpha r_{i}),r_{i}^{-1}d_{X},x), (B^{C(\mathbb{S}^{3}/\Gamma)}(\overline{x},\alpha),d_{C(\mathbb{S}^{3}/\Gamma)},\overline{x}) \right) <\epsilon_i$$
for some $\overline{x}\in C(\mathbb{S}^{3}/\Gamma)$ with $\frac{1}{4}\leq d(c_{0},\overline{x})\leq4$.
Suppose by way of contradiction that there exist $\epsilon>0$ and
$y_{i}\in B^{X}(x^{\ast},2r_{i})\setminus\overline{B}^{X}(x^{\ast},\frac{1}{2}r_{i})$
where $$d_{GH}((B^{X}(y_{i},\alpha r_{i}),r_{i}^{-1}d_{X},y_{i}),(B^{C(\mathbb{S}^{3}/\Gamma)}(\overline{x},\alpha),d_{C(\mathbb{S}^{3}/\Gamma)},\overline{x}))>\epsilon$$ 
for every $\overline{x}\in C(\mathbb{S}^{3}/\Gamma)$ with $\frac{1}{4}\leq d(c_0,\overline{x})\leq4$.
Let $\psi_{i}:B^{X}(x^{\ast},100r_{i})\to C(\mathbb{S}^{3}/\Gamma)$
be $\delta_{i}$-Gromov-Hausdorff approximations $$(B^{X}(x^{\ast},100r_{i}),r_{i}^{-1}d_{X},x^{\ast})\to(B^{C(\mathbb{S}^{3}/\Gamma)}(c_{0},100),d_{C(\mathbb{S}^{3}/\Gamma)},c_{0}),$$
where $\delta_{i}\to0$. Then $\psi_{i}|B^{X}(y_{i},\alpha r_{i})$
is a $\frac{1}{2}\delta_{i}$-Gromov-Hausdorff approximation from
$$(B^{X}(y_{i},\alpha r_{i}),r_{i}^{-1}d_{X},y_{i}) \to (B^{C(\mathbb{S}^{3}/\Gamma)}(\psi_{i}(y_{i}),\alpha),d_{C(\mathbb{S}^{3}/\Gamma)},\psi_{i}(y_{i})),$$
where $\frac{1}{3}\leq d(c_{0},\psi_{i}(y_{i}))\leq3$. Passing to
a subsequence, we may assume that $\psi_{i}(y_{i})$ converges in
to some $y_{\infty}\in C(\mathbb{S}^{3}/\Gamma)$ with $\frac{1}{3}\leq d(c_{0},y_{\infty})\leq3$.
Then $(B^{X}(y_{i},\alpha r_{i}),r_{i}^{-1}d_{X},y_{i})$ converges
in the pointed Gromov-Hausdorff sense to $(B^{C(\mathbb{S}^3/\Gamma)}(y_{\infty},\alpha),d_{C(\mathbb{S}^{3}/\Gamma)},y_{\infty})$,
a contradiction.

Because $|\Gamma|\leq N(A)$, we can choose $\alpha =\alpha(A)>0$ sufficiently small so that $|B^{C(\mathbb{S}^3/\Gamma)}(y,\alpha)| = \omega_n \alpha^n$ for all $y\in C(\mathbb{S}^3/\Gamma)$ with $\frac{1}{4}\leq d(y,c_0)\leq 4$. Another application of Theorem 2.37 of \cite{tianfano} tells us that the volume ratios $(\alpha r_{i})^{-n}|B^{X}(x,\alpha r_{i})|$ converge
to the Euclidean volume ratio, locally uniformly for $x\in X$ with
$2r_{i}>d_{X}(x,x^{\ast})>\frac{1}{2}r_{i}$. Note that we could also perform a conformal change using the potential function $f$ as in \cite{zhangdegen}, and appeal to the volume convergence theorem for manifolds with Ricci curvature bounded below. By the
$Y(A)$-regularity of the singular space $\mathcal{X}$, we may conclude
that $r_{Rm}^{X}(x_{i})>Y^{-1}r_{i}$ for some $Y=Y(A)<\infty$. In
particular, also using the lower volume bound, we may assume by passing
to a subsequence and applying the Cheeger-Gromov compactness theorem,
 that the convergence of the rescaled metrics $(B^{X}(x_{i},\frac{1}{2}\alpha_{i}r_{i}),r_{i}^{-2}g)$
to a ball in $B^{C(\mathbb{S}^{3}/\Gamma)}(c_{0},4)\setminus B^{C(\mathbb{S}^{3}/\Gamma)}(c_{0},\frac{1}{4})$
is smooth. However, $C(\mathbb{S}^{3}/\Gamma)$ is flat away from
the cone point, so actually $|Rm(x_{i})|d_{X}^{2}(x_{i},x^{\ast})\to0$
as $i\to\infty$, a contradiction. We may therefore conclude that
$|Rm|(x)d_X^{2}(x,x^{\ast})\leq\epsilon(x)$ for $x\in B^{X}(x^{\ast},\delta)\setminus\{x^{\ast}\}$,
where $\lim_{x\to x^{\ast}}\epsilon(x)=0$. By local estimates for shrinking GRS (c.f. the proof of Theorem 2.5 in \cite{hasl}), we even have $|\nabla^{k}Rm|(x)d_X^{2+k}(x,x^{\ast})\leq\epsilon(x)$
for $x\in B^{X}(x^{\ast},\delta)$, any $k\geq0$.\\

\noindent \textbf{Claim: }$C(\mathbb{S}^{3}/\Gamma)$ is the unique
tangent cone of $X$ at $x^{\ast}$.

We proceed as in the proof of Lemma 5.13 in \cite{nakajima}. The above arguments give that, for any tangent cone $C(\mathbb{S}/\Gamma')$ of $X$
at $x^{\ast}$, there is a sequence $s_{i}\searrow0$ such that $B^{X}(x^{\ast},2s_{i})\setminus B^{X}(x^{\ast},s_{i})$
is diffeomorphic to $(1,2)\times\mathbb{S}^{3}/\Gamma'$. It therefore
suffices to find $\overline{r}>0$ such that $B^{X}(x^{\ast},2r)\setminus\overline{B}^{X}(x^{\ast},r)$
and $B^{X}(x^{\ast},2s)\setminus\overline{B}^{X}(x^{\ast},s)$ are
diffeomorphic for all $s,r\in(0,\overline{r})$. Set 
\begin{align*}
\theta(r)&:=\sup\{\angle(\dot{\gamma}(r),\dot{\eta}(r));\gamma,\eta:[0,r]\to B^{X}(x,r) \\&\hspace{ 20 mm}\text{ are unit-speed minimizing geodesics from }x^{\ast}\text{ to some }x\in\partial B^{X}(x,r)\}.
\end{align*}
We claim that $\lim_{r\to0}\theta(r)=0$. Otherwise, there exists
$\tau>0$ and a sequence of unit-speed geodesics $\gamma_{i},\eta_{i}:[0,r_{i}]\to X$ from $x$ to some $x_{i}\in\partial B^{X}(x^{\ast},r_{i})$, such that $\angle(\dot{\gamma}(r_{i}),\dot{\eta}_{i}(r_{i}))\geq\tau$.
After passing to a subsequence, we have pointed Cheeger-Gromov convergence
$$(B^{X}(x^{\ast},\alpha r_{i}),r_{i}^{-2}g,x_{i})\to(B^{C(\mathbb{S}^{3}/\Gamma')}(x_{\infty},\alpha),g_{C(\mathbb{S}^{3}/\Gamma')},x_{\infty})$$
for some $x_{\infty}\in\partial B^{C(\mathbb{S}^{3}/\Gamma')}(c_{0},1)$, $$r_i^{-1} \dot{\gamma }_i(r_i)\to v \in T_{x_\infty}C(\mathbb{S}^3/\Gamma '),$$ $$r_i^{-1}\dot{\eta}_i(r_i)\to w \in T_{x_\infty}C(\mathbb{S}^3/\Gamma'),$$
pointed Gromov-Hausdorff convergence $$(B^{X}(x^{\ast},\delta),r_{i}^{-1}d_{X},x^{\ast})\to(C(\mathbb{S}^{3}/\Gamma'),g_{C(\mathbb{S}^{3}/\Gamma')},c_{0})$$
and (after constant-speed reparametrization) $\gamma_{i},\eta_{i}$
converge (smoothly on $(0,1]$) to unit-speed geodesics $\gamma_{\infty},\eta_{\infty}:[0,1]\to X$
from $c_{0}$ to $x_{\infty}$ with $$\angle(\dot{\gamma}_{\infty}(1),\dot{\eta}_{\infty}(1)) = \angle(v,w) \geq\tau,$$
contradicting the fact that there is a unique minimizing geodesic
from $c_{0}$ to any $x\in C(\mathbb{S}^{3}/\Gamma'$). Therefore
$\lim_{r\to0}\theta(r)=0$, so for $r_{0}>0$ sufficiently small,
we can construct a smooth vector field whose flow can be used to construct a homeomorphism (see Proposition 12.1.2 and Lemma 12.1.3 of \cite{peter}) between $B^{X}(x^{\ast},2s_{i})\setminus B^{X}(x^{\ast},s_{i})$,
$i=1,2$, for any $s_{1},s_{2}\in(0,r_{0}]$. However, we know $B^{X}(x^{\ast},2s_{i})\setminus B^{X}(x^{\ast},s_{i})$
is homeomorphic to $(1,2)\times\mathbb{S}^{3}/\Gamma_{i}$ for some
finite subgroups, so we must have $\mathbb{S}^3/\Gamma_{1}=\mathbb{S}^3/\Gamma_{2}$. $\square$

We can now apply the argument in Step 1 of \cite{donaldsun} (see also
\cite{BZ1}, \cite{nakajima},\cite{tianchern}) verbatim to our situation
to conclude that there is a diffeomorphism $F:(B(0^{4},r_{0})\setminus\{0^{4}\})/\Gamma\to B^{X}(x^{\ast},r_{0})\setminus\{x^{\ast}\}$
such that $(F\circ\pi)^{\ast}g$ extends to a $C^{0}$ Riemannian
metric on $B(0^{4},r_{0})$, where $\pi:\mathbb{R}^{4}\to\mathbb{R}^{4}/\Gamma$
is the quotient map. By replacing $g$ with $(F\circ \pi)^{\ast}g$, we may as well assume $\Gamma$ is trivial. Define $A(r_{1},r_{2}):=B^{X}(x^{\ast},r_{2})\setminus\overline{B}^{X}(x^{\ast},r_{1})$
and $B^{\ast}:=B^{X}(x^{\ast},r_{0})\setminus\{x^{\ast}\}$.

\noindent \textbf{Claim: }$\int_{B^{\ast}}|Rm|^{2}dg<\infty$.

In four dimensions, the curvature tensor $Rm$ admits the orthogonal
decomposition 
\begin{equation} \label{decomp} Rm=\dfrac{R}{24}g\owedge g+\frac{1}{2}\left(Rc-\frac{R}{4}g\right)\owedge g+W. \end{equation}
Because $|R|\leq A$ on $\mathcal{R}$, the first term of (\ref{decomp}) is bounded
pointwise. We use the method of \cite{hasl} to estimate the second
term of (\ref{decomp}). Fix $\beta\in(0,1),$and let $\phi\in C_{c}^{\infty}(A(\beta r_{0},r_{0}))$
be a cutoff function with $|\nabla\phi|\leq C(n)(\beta r_{0})^{-1}$ on $A(\beta r_{0},2\beta r_{0})$,
$|\nabla\phi|\leq C(n)r_{0}^{-1}$ on $A(\frac{1}{2}r_{0},2r_{0})$, and
$\phi=1$ on $A(2\beta r_{0},\frac{1}{2}r_{0})$. Then, setting $E:=\sup_{B}(e^{-f}+|\nabla f|)$,
we get
\begin{align*}
\int_{B^{\ast}}|Rc|^{2}\phi^{2}e^{-f}dg= & \int_{B^{\ast}}\left\langle \frac{1}{2}g-\nabla^{2}f,Rc\right\rangle \phi^{2}e^{-f}dg\\
= & C(A,E)+\int_{B^{\ast}}\langle\nabla f,\text{div}_{f}(\phi^{2}Rc)\rangle e^{-f}dg\\
\leq & C(A,E)+\int_{B^{\ast}}2|\nabla f|\cdot|\nabla\phi|\cdot\phi|Rc|e^{-f}dg\\
\leq & C(A,E)+\frac{1}{2}\int_{B^{\ast}}|Rc|^{2}\phi^{2}e^{-f}dg+2\int_{B^{\ast}}|\nabla f|^{2}|\nabla\phi|^{2}e^{-f}dg,
\end{align*}
since $\text{div}_{f}Rc=0$. Rearranging, we conclude
\begin{align*}
\int_{B^{\ast}}|Rc|^{2}\phi^{2}e^{-f}dg\leq & C(A,E)+C(A,E,r_{0})\beta^{-2}\text{Vol}_{g}(A(0,2\beta r_{0}))\\
\leq & C(A,E)+C(A,E,r_{0})\beta^{2}.
\end{align*}
Taking $\beta\to 0$, and recalling that $f$ is locally bounded above, we obtain $\int_{B^{\ast}}|Rc|^{2}dg<\infty$.
Finally, to estimate the third term of (\ref{decomp}), we further decompose $W$ into
the self-dual and anti-self-dual parts $W_{\pm}$, and then employ the strategy of \cite{donaldsun}. Let $\mathcal{A}_{+}$ be the connection on the bundle $\Lambda_{+}$ of self-dual forms on
$\mathcal{R}$ induced by the Levi-Civita connection of $(\mathcal{R},g)$ (see section 6.D of \cite{besse} for definitions).
Then, because $W_+$ is self-dual, $(Rc-\frac{R}{4}g)\owedge g$ is anti-self-dual, and $\Lambda_+,\Lambda_-$ are orthogonal, we have
$$
\int_{A(s,r)}\left(|W_{+}|^{2} + \dfrac{R^2}{12} - \left| Rc-\frac{R}{4}g \right|^2 \right) dg=  \int_{A(s,r)}tr(F_{\mathcal{A}_{+}}\wedge F_{\mathcal{A}_{+}}),
$$
but also (by Example 2.5 of \cite{cheegersimons})
\[
\int_{A(s,r)}tr(F_{\mathcal{A}_{+}}\wedge F_{\mathcal{A}_{+}})=CS(\mathcal{A}_{+},\partial B(x^{\ast},r))-CS(\mathcal{A}_{+},\partial B(x^{\ast},s))\hspace{6mm} (\text{mod }\mathbb{Z}),
\]
where 
\[
CS(B,\Sigma):=\frac{1}{8\pi^2} \int_{\Sigma}tr\left(dB\wedge B+\frac{2}{3}B\wedge B\wedge B\right)\in\mathbb{R}/\mathbb{Z}.
\]
is the Chern-Simons invariant (associated to the first Pontryagin class) of a connection $\nabla=d+B$ on a trivial bundle over a 3-manifold $\Sigma$, once we have chosen an arbitrary global section of the bundle. However, by the Cheeger-Gromov convergence
of $r^{-2}g|\partial B(x^{\ast},r)$ to a flat bundle metric on $(T\mathbb{R}^4)|\mathbb{S}^{3}$ as $r\to \infty$,
we may conclude that $\mathcal{A}_{+}|\partial B(x^{\ast},r)$ converge (after
pulling back by diffeomorphisms $\psi_i:\mathbb{S}^3 \to \partial B(x^{\ast},r)$) to the Euclidean connection $D$ on the trivial bundle of self-dual 2-forms of $(\mathbb{R}^{4}\setminus\{0\})$ restricted to $\mathbb{S}^3$, which has Chern-Simons invariant $0$. This means
$$CS(\mathcal{A}_+,\partial B(x^{\ast},r))=CS(\psi_i^{\ast}\mathcal{A}_+,\mathbb{S}^3)\to CS(D,\mathbb{S}^3)=0$$
as $r\searrow 0$ , so we can choose $r\in(0,r_{0}]$ sufficiently small such that $|CS(\mathcal{A}_{+},\partial B(x^{\ast},s))|\leq\frac{1}{8}$
mod $\mathbb{Z}$ for all $s\in[0,r]$. Because $$s\mapsto\int_{A(s,r)}tr(F_{\mathcal{A}_{+}}\wedge F_{\mathcal{A}_{+}})\in[-\frac{1}{4},\frac{1}{4}]\hspace{6 mm} (\text{mod } \mathbb{Z})$$
is continuous, we conclude that the integral is bounded uniformly
(in $\mathbb{R}$) for all $s<r$. In particular, we can take $s\searrow0$ to obtain
$\int_{B^{\ast}}|W_{+}|^{2}dg<\infty$, and the proof of $\int_{B^{\ast}}|W_{-}|^{2}dg<\infty$
is similar. $\square$

We can now argue as in \cite{caosesum}, \cite{tianchern}, to conclude that
in fact $B^{\ast}$ has the structure of a $C^{\infty}$ orbifold
at $x^{\ast}$. Note that, because we have bounds on $f,|\nabla f|$
on $B^{\ast},$the only difference in our setting is that we must
use the $\epsilon$-regularity theorem that is
Theorem 1.1 of \cite{jiangreg} or Theorem 1.2 of \cite{jiangreg} (note that the completeness condition can be replaced with the condition that a larger geodesic ball is locally compact).
Also, $R+|\nabla f|^{2}=f-\mathcal{W}(g,f)$, $|R|\leq A$,
and the quadratic growth of $f$ imply that all critical points of $f$ must occur in some bounded set. On the other hand, any orbifold point of $X$
must be a critical point: if $\varphi:\mathbb{R}^{4}/\Gamma\supseteq U\to B^{X}(x,\delta)$
is an orbifold chart, and $\pi:\mathbb{R}^{4}\to\mathbb{R}^{4}/\Gamma$
is the quotient map, then $\nabla^{(\pi\circ\varphi)^{\ast}g}(\pi\circ\varphi)^{\ast}f$
must be fixed by all of $\Gamma$, so must be the zero vector. Since
$X\setminus\mathcal{R}$ is discrete and bounded, it must be finite. 
\end{proof}

\noindent \emph{Proof of Theorem 3.} This is immediate from \thref{removable}. $\square$

\section{Appendix}

In this section, we give further details for the claim in Lemma \ref{tech} that 
$$|B^X(y_i,\alpha \lambda_i^{-1})\cap \mathcal{R}|_g \geq (\omega_n -\epsilon_i )(\alpha \lambda_i^{-1})^4$$ for some sequence $\epsilon_i \to 0$. The main idea is to use the fact that, for $i\in \mathbb{N}$ sufficiently large, $(B^X(y_i,\alpha \lambda_i^{-1}),\lambda_i d,y_i)$ is arbitrarily close to a Euclidean ball in the pointed Gromov-Hausdroff sense, and to then appeal to a volume convergence theorem for Riemannian manifolds with integral Ricci lower bounds.

Observe that, by Lemma 6.1 of \cite{BZ1}, we have
$$|Rc|_{\widetilde{g}^i}(\cdot,0)\leq C(A)(r_{Rm}^{\widetilde{g}^i})^{-1}(\cdot,0),$$
so combining this with the integral estimate for the curvature scale
(Theorem 1.7 of \cite{bam2}) gives
\[
\int_{B_{\widetilde{g}^{i}}(x_{i},0,1)}|Rc|^{3}(\cdot ,0) d\widetilde{g}_{0}^{i}\leq\int_{B_{\widetilde{g}^{i}}(x_{i},0,1)}(r_{Rm}^{g^i}(\cdot,0))^{-3}d\widetilde{g}_{0}^{i}\leq C(A).
\]
Note that we actually have a local $L^p$ bound for $Rc$ for any $p<4$, and the following arguments will work for any $p\in (2,4)$, but we choose $p=3$ for convenience. 

Let $\mathcal{H}_d^4=\mathcal{H}^4$ be the $4$-dimensional Hausdorff measure on the metric space $(X,d)$. Because $\mathcal{H}^{4}(X\setminus\mathcal{R})=0$, and because $\mathcal{H}^{4}$ agrees with the Riemannian volume measure on any 4-dimensional Riemannian manifold (in particular, on $\mathcal{R})$, we have $\mathcal{H}^{4}(S)=|S\cap\mathcal{R}|$ for any subset $S\subseteq X$. Thus
$$\mathcal{H}_{\lambda_i d}^4(B^X(y_i,\alpha \lambda_i^{-1}))=\lambda_i^4 |B^X(y_i,\alpha \lambda_i^{-1})\cap \mathcal{R}|_g,$$
$$\mathcal{H}_{d_Z}^4(B^Z(y_{\infty},\alpha))=\omega_n \alpha^n.$$

We now restate the modification of Theorem 2.37 of \cite{tianfano} that we will be using. Denote by $|Rc_-|(x)$ the absolute value of the smallest negative eigenvalue of $Rc(x)$ (if $Rc(x)\geq 0$, then $|Rc_-|=0$). 
\begin{lem} \label{aux} For any $\kappa>0$, $\Lambda<\infty$, $n\in \mathbb{N}$, and $p>n$, there exist $r_0=r_0(n,p,\kappa,\Lambda,\epsilon)>0$ such that the following holds. Suppose $(M_i^n,g_i,x_i)$ is a sequence of complete Riemannian manifolds satisfying:\\
$(i)$ $\int_{B(x,1)}|Rc_-|^p dg \leq \Lambda$ for all $x\in M_i$, \\
$(ii)$ $|B(x,r)|\geq \kappa r^n$ for all $r\in (0,1]$, $x\in M$.\\
Assume that $(M_i^n,g_i,x_i)$ converge in the pointed Gromov-Hausdorff sense to the complete metric length space $(X,d,p)$. Then, for any $r\in (0,r_0]$, we have $$\mathcal{H}_d^n(B(x,r))=\lim_{i\to \infty} |B(x_i,r_i)|.$$ 
\end{lem}
The difference between this lemma and Theorem 2.37 of \cite{tianfano} is that we only require a local integral Ricci bound $(i)$ rather than the global bound $$\int_M |Rc_-|^p dg \leq \Lambda$$
assumed in \cite{tianfano}. However, in \cite{tianfano}, the objects under consideration are time slices of a normalized Ricci flow on a Fano threefold, which have uniformly bounded diameter. The proof of Theorem 2.37 is stated to be a modification of volume convergence for noncollapsed Riemannian manifolds with Ricci curvature bounded below, given in \cite{volumecold,warpedrigid}. A careful examination of the proof shows that only the conditions $(i),(ii)$ are used, essentially due to the fact that the involved arguments are all local.

The following elementary lemma is essentially a consequence of Lemma 22 and a diagonal argument.
\begin{lem} Let $(X_{k},d_{k},p_{k})$ be a sequence of limit spaces as in Lemma 22, converging in the pointed Gromov-Hausdorff sense to (X,d,p), and suppose $r\leq r_{0}(n,p,\kappa,\Lambda)$. Then $$\mathcal{H}^{n}(B(p_{k},r))\to\mathcal{H}^{n}(B(p,r)).$$\end{lem}
\begin{proof} For each $k\in\mathbb{N}$, let $(M_{k,i},g_{k,i},x_{k,i})$ be a sequence of complete, pointed Riemannian manifolds satisfying $(i),(ii)$ of Lemma \ref{aux}, which converge in the pointed Gromov-Hausdorff sense to $(X_{k},d_{k},x_{k})$ as $i\to\infty$. Also let $(M_{i},g_{i},x_{i})$ be a sequence of such manifolds converging to $(X,d,p)$ in the pointed Gromov-Hausdorff sense. By Lemma \ref{aux}, we know that $$\lim_{i\to\infty}|B(x_{k,i},r)|_{g_{k,i}}=\mathcal{H}^{n}(B(x_{k},r))$$for each $k\in\mathbb{N}$. Thus, for each $k\in\mathbb{N}$, we can find $i(k)\in\mathbb{N}$ such that $$\left| |B(x_{k,i(k)},r)|_{g_{k,i(k)}}-\mathcal{H}^{n}(B(x_{k},r))\right| \leq2^{-k},$$
$$d_{GH}\left( (B(x_{k,i(k)},\alpha_k r),d_{g_{k,i(k)}},x_{k,i(k)}),(B(x_{k},\alpha_k r),d_{k},x_{k})\right)\leq r2^{-k},$$
where $\alpha_{k}\to\infty$. In particular, $(M_{k,i(k)},g_{k,i(k)},x_{k,i(k)})$ converge in the pointed Gromov-Hausdorff sense to $(X,d,p)$, so $$\left| |B(x_{k,i(k)},r)|_{g_{k,i(k)}}-\mathcal{H}^{n}(B(x,r))\right|\to 0.$$Combining expressions gives the claim.
\end{proof}

After possibly shrinking $\alpha$ so that $\alpha<r_0$, we can apply the previous lemma to 
$$(B^X(y_i,\alpha \lambda_i^{-1}),\lambda_i d,y_i)\to(B^Z(y_{\infty},\alpha),d_Z,y_{\infty}),$$
we conclude that
$$\mathcal{H}_{\lambda_i d}^4(B^X(y_i,\alpha \lambda_i^{-1})\to \mathcal{H}^4(B^Z(y_{\infty},\alpha))$$
as $i\to \infty$. That is,
$$\lim_{i\to\infty}\lambda_i^4 |B^X(y_i,\alpha \lambda_i^{-1})\cap \mathcal{R}|=\omega_n \alpha^4 ,$$
so the claim follows.

\printbibliography

@article {bam1,
    AUTHOR = {Bamler, Richard H.},
     TITLE = {Structure theory of singular spaces},
   JOURNAL = {J. Funct. Anal.},
  FJOURNAL = {Journal of Functional Analysis},
    VOLUME = {272},
      YEAR = {2017},
    NUMBER = {6},
     PAGES = {2504--2627},
      ISSN = {0022-1236},
   MRCLASS = {53C23 (53C21 53C25)},
  MRNUMBER = {3603307},
MRREVIEWER = {Nan Li},
       DOI = {10.1016/j.jfa.2016.10.020},
       URL = {https://doi.org/10.1016/j.jfa.2016.10.020},
}

@ARTICLE{bam2,
   author = {{Bamler}, R.~H.},
    title = "{Convergence of Ricci flows with bounded scalar curvature}",
  journal = {ArXiv e-prints},
archivePrefix = "arXiv",
   eprint = {1603.05235},
 primaryClass = "math.DG",
     year = 2016,
    month = mar,
   adsurl = {http://adsabs.harvard.edu/abs/2016arXiv160305235B}
}

@article {bingspace1,
    AUTHOR = {Chen, Xiuxiong and Wang, Bing},
     TITLE = {Space of {R}icci flows {I}},
   JOURNAL = {Comm. Pure Appl. Math.},
  FJOURNAL = {Communications on Pure and Applied Mathematics},
    VOLUME = {65},
      YEAR = {2012},
    NUMBER = {10},
     PAGES = {1399--1457},
      ISSN = {0010-3640},
   MRCLASS = {53C44 (14J45 35K55)},
  MRNUMBER = {2957704},
MRREVIEWER = {Julien Keller},
       DOI = {10.1002/cpa.21414},
       URL = {https://doi-org.proxy.library.cornell.edu/10.1002/cpa.21414},
}

@misc{bingspace2,
    title={Space of Ricci flows (II)},
    author={Xiuxiong Chen and Bing Wang},
    year={2014},
    eprint={1405.6797},
    archivePrefix={arXiv},
    primaryClass={math.DG}
}

@book {besse,
    AUTHOR = {Besse, Arthur L.},
     TITLE = {Einstein manifolds},
    SERIES = {Classics in Mathematics},
      NOTE = {Reprint of the 1987 edition},
 PUBLISHER = {Springer-Verlag, Berlin},
      YEAR = {2008},
     PAGES = {xii+516},
      ISBN = {978-3-540-74120-6},
   MRCLASS = {53C25 (53-02)},
  MRNUMBER = {2371700},
}

@article {BZ1,
    AUTHOR = {Bamler, Richard H. and Zhang, Qi S.},
     TITLE = {Heat kernel and curvature bounds in {R}icci flows with bounded
              scalar curvature},
   JOURNAL = {Adv. Math.},
  FJOURNAL = {Advances in Mathematics},
    VOLUME = {319},
      YEAR = {2017},
     PAGES = {396--450},
      ISSN = {0001-8708},
   MRCLASS = {53C44 (58J35)},
  MRNUMBER = {3695879},
MRREVIEWER = {Asuka Takatsu},
       DOI = {10.1016/j.aim.2017.08.025},
       URL = {https://doi.org/10.1016/j.aim.2017.08.025},
}

@ARTICLE{BZ2,
   author = {{Bamler}, R.~H. and {Zhang}, Q.~S.},
    title = "{Heat kernel and curvature bounds in Ricci flows with bounded scalar curvature --- Part II}",
  journal = {ArXiv e-prints},
archivePrefix = "arXiv",
   eprint = {1506.03154},
 primaryClass = "math.DG",
     year = 2015,
    month = jun,
   adsurl = {http://adsabs.harvard.edu/abs/2015arXiv150603154B}
}

@article{cao,
author = {Xiaodong Cao and Qi S. Zhang},
journal = {Advances in Mathematics},
number = {5},
pages = {2891-919},
title = {The Conjugate Heat Equation and Ancient Solutions of the Ricci Flow},
volume = {228},
year = {2011},
}

@article {caosesum,
    AUTHOR = {Cao, Huai-Dong and Sesum, N.},
     TITLE = {A compactness result for {K}\"{a}hler {R}icci solitons},
   JOURNAL = {Adv. Math.},
  FJOURNAL = {Advances in Mathematics},
    VOLUME = {211},
      YEAR = {2007},
    NUMBER = {2},
     PAGES = {794--818},
      ISSN = {0001-8708},
   MRCLASS = {53C44 (53C55)},
  MRNUMBER = {2323545},
MRREVIEWER = {Oliver C. Schn\"{u}rer},
       DOI = {10.1016/j.aim.2006.09.011},
       URL = {https://doi.org/10.1016/j.aim.2006.09.011},
}

@incollection {cheegersimons,
    AUTHOR = {Cheeger, Jeff and Simons, James},
     TITLE = {Differential characters and geometric invariants},
 BOOKTITLE = {Geometry and topology ({C}ollege {P}ark, {M}d., 1983/84)},
    SERIES = {Lecture Notes in Math.},
    VOLUME = {1167},
     PAGES = {50--80},
 PUBLISHER = {Springer, Berlin},
      YEAR = {1985},
   MRCLASS = {53C20 (57R20 58G10)},
  MRNUMBER = {827262},
MRREVIEWER = {P. Molino},
       DOI = {10.1007/BFb0075216},
       URL = {https://doi-org.proxy.library.cornell.edu/10.1007/BFb0075216},
}

@book {chowbook3,
    AUTHOR = {Chow, Bennett and Chu, Sun-Chin and Glickenstein, David and
              Guenther, Christine and Isenberg, James and Ivey, Tom and
              Knopf, Dan and Lu, Peng and Luo, Feng and Ni, Lei},
     TITLE = {The {R}icci flow: techniques and applications. {P}art {III}.
              {G}eometric-analytic aspects},
    SERIES = {Mathematical Surveys and Monographs},
    VOLUME = {163},
 PUBLISHER = {American Mathematical Society, Providence, RI},
      YEAR = {2010},
     PAGES = {xx+517},
      ISBN = {978-0-8218-4661-2},
   MRCLASS = {53C44 (35K08 35K55)},
  MRNUMBER = {2604955},
MRREVIEWER = {James Alexander McCoy},
       DOI = {10.1090/surv/163},
       URL = {https://doi-org.proxy.library.cornell.edu/10.1090/surv/163},
}

@article {collinstosatti,
    AUTHOR = {Collins, Tristan C. and Tosatti, Valentino},
     TITLE = {K\"{a}hler currents and null loci},
   JOURNAL = {Invent. Math.},
  FJOURNAL = {Inventiones Mathematicae},
    VOLUME = {202},
      YEAR = {2015},
    NUMBER = {3},
     PAGES = {1167--1198},
      ISSN = {0020-9910},
   MRCLASS = {32J25 (32Q15 32U40 53C55)},
  MRNUMBER = {3425388},
MRREVIEWER = {Jian Xiao},
       DOI = {10.1007/s00222-015-0585-9},
       URL = {https://doi-org.proxy.library.cornell.edu/10.1007/s00222-015-0585-9},
}

@article {donaldsun,
    AUTHOR = {Donaldson, Simon and Sun, Song},
     TITLE = {Gromov-{H}ausdorff limits of {K}\"{a}hler manifolds and algebraic
              geometry},
   JOURNAL = {Acta Math.},
  FJOURNAL = {Acta Mathematica},
    VOLUME = {213},
      YEAR = {2014},
    NUMBER = {1},
     PAGES = {63--106},
      ISSN = {0001-5962},
   MRCLASS = {53C55 (32Q20 53C23)},
  MRNUMBER = {3261011},
MRREVIEWER = {Valentino Tosatti},
       DOI = {10.1007/s11511-014-0116-3},
       URL = {https://doi-org.proxy.library.cornell.edu/10.1007/s11511-014-0116-3},
}

@article {hasl,
    AUTHOR = {Haslhofer, Robert and M\"{u}ller, Reto},
     TITLE = {A compactness theorem for complete {R}icci shrinkers},
   JOURNAL = {Geom. Funct. Anal.},
  FJOURNAL = {Geometric and Functional Analysis},
    VOLUME = {21},
      YEAR = {2011},
    NUMBER = {5},
     PAGES = {1091--1116},
      ISSN = {1016-443X},
   MRCLASS = {53C21 (53C23 53C25 53C44)},
  MRNUMBER = {2846384},
MRREVIEWER = {Yuguang Zhang},
       DOI = {10.1007/s00039-011-0137-4},
       URL = {https://doi-org.proxy.library.cornell.edu/10.1007/s00039-011-0137-4},
}

@article {hein,
    AUTHOR = {Hein, Hans-Joachim and Naber, Aaron},
     TITLE = {New logarithmic {S}obolev inequalities and an
              {$\epsilon$}-regularity theorem for the {R}icci flow},
   JOURNAL = {Comm. Pure Appl. Math.},
  FJOURNAL = {Communications on Pure and Applied Mathematics},
    VOLUME = {67},
      YEAR = {2014},
    NUMBER = {9},
     PAGES = {1543--1561},
      ISSN = {0010-3640},
   MRCLASS = {35A23 (35R01 53C23 53C44)},
  MRNUMBER = {3245102},
MRREVIEWER = {Yu Ding},
       DOI = {10.1002/cpa.21474},
       URL = {https://doi.org/10.1002/cpa.21474},
}

@article {shaosai,
    AUTHOR = {Huang, Shaosai},
     TITLE = {{$\varepsilon$}-regularity and structure of four-dimensional
              shrinking {R}icci solitons},
   JOURNAL = {Int. Math. Res. Not. IMRN},
  FJOURNAL = {International Mathematics Research Notices. IMRN},
      YEAR = {2020},
    NUMBER = {5},
     PAGES = {1511--1574},
      ISSN = {1073-7928},
   MRCLASS = {53C25},
  MRNUMBER = {4073948},
       DOI = {10.1093/imrn/rny069},
       URL = {https://doi-org.proxy.library.cornell.edu/10.1093/imrn/rny069},
}

@article {jiangreg,
    AUTHOR = {Ge, Huabin and Jiang, Wenshuai},
     TITLE = {{$\epsilon$}-regularity for shrinking {R}icci solitons and
              {R}icci flows},
   JOURNAL = {Geom. Funct. Anal.},
  FJOURNAL = {Geometric and Functional Analysis},
    VOLUME = {27},
      YEAR = {2017},
    NUMBER = {5},
     PAGES = {1231--1256},
      ISSN = {1016-443X},
   MRCLASS = {53C25 (53C44)},
  MRNUMBER = {3714720},
MRREVIEWER = {Gurupadavva Ingalahalli},
       DOI = {10.1007/s00039-017-0420-0},
       URL = {https://doi-org.proxy.library.cornell.edu/10.1007/s00039-017-0420-0},
}

@article {kahler,
    AUTHOR = {Sesum, Natasa and Tian, Gang},
     TITLE = {Bounding scalar curvature and diameter along the {K}\"{a}hler
              {R}icci flow (after {P}erelman)},
   JOURNAL = {J. Inst. Math. Jussieu},
  FJOURNAL = {Journal of the Institute of Mathematics of Jussieu. JIMJ.
              Journal de l'Institut de Math\'{e}matiques de Jussieu},
    VOLUME = {7},
      YEAR = {2008},
    NUMBER = {3},
     PAGES = {575--587},
      ISSN = {1474-7480},
   MRCLASS = {53C44 (53C55)},
  MRNUMBER = {2427424},
MRREVIEWER = {Julien Keller},
       DOI = {10.1017/S1474748008000133},
       URL = {https://doi.org/10.1017/S1474748008000133},
}

@article {naber,
    AUTHOR = {Naber, Aaron},
     TITLE = {Noncompact shrinking four solitons with nonnegative curvature},
   JOURNAL = {J. Reine Angew. Math.},
  FJOURNAL = {Journal f\"ur die Reine und Angewandte Mathematik. [Crelle's
              Journal]},
    VOLUME = {645},
      YEAR = {2010},
     PAGES = {125--153},
      ISSN = {0075-4102},
   MRCLASS = {53C25 (53C21)},
  MRNUMBER = {2673425},
MRREVIEWER = {Esther Cabezas Rivas},
       DOI = {10.1515/CRELLE.2010.062},
       URL = {https://doi.org/10.1515/CRELLE.2010.062},
}

@article {mant,
    AUTHOR = {Mantegazza, Carlo and M\"uller, Reto},
     TITLE = {Perelman's entropy functional at {T}ype {I} singularities of
              the {R}icci flow},
   JOURNAL = {J. Reine Angew. Math.},
  FJOURNAL = {Journal f\"ur die Reine und Angewandte Mathematik. [Crelle's
              Journal]},
    VOLUME = {703},
      YEAR = {2015},
     PAGES = {173--199},
      ISSN = {0075-4102},
   MRCLASS = {53C44},
  MRNUMBER = {3353546},
MRREVIEWER = {Anqiang Zhu},
       DOI = {10.1515/crelle-2013-0039},
       URL = {https://doi.org/10.1515/crelle-2013-0039},
}

@article {nakajima,
    AUTHOR = {Bando, Shigetoshi and Kasue, Atsushi and Nakajima, Hiraku},
     TITLE = {On a construction of coordinates at infinity on manifolds with
              fast curvature decay and maximal volume growth},
   JOURNAL = {Invent. Math.},
  FJOURNAL = {Inventiones Mathematicae},
    VOLUME = {97},
      YEAR = {1989},
    NUMBER = {2},
     PAGES = {313--349},
      ISSN = {0020-9910},
   MRCLASS = {53C20 (53C25)},
  MRNUMBER = {1001844},
MRREVIEWER = {Thomas H. Otway},
       DOI = {10.1007/BF01389045},
       URL = {https://doi.org/10.1007/BF01389045},
}

@article {penglu,
    AUTHOR = {Lu, Peng},
     TITLE = {A local curvature bound in {R}icci flow},
   JOURNAL = {Geom. Topol.},
  FJOURNAL = {Geometry \& Topology},
    VOLUME = {14},
      YEAR = {2010},
    NUMBER = {2},
     PAGES = {1095--1110},
      ISSN = {1465-3060},
   MRCLASS = {53C44},
  MRNUMBER = {2629901},
MRREVIEWER = {Esther Cabezas Rivas},
       DOI = {10.2140/gt.2010.14.1095},
       URL = {https://doi-org.proxy.library.cornell.edu/10.2140/gt.2010.14.1095},
}

@ARTICLE{perl,
   author = {{Perelman}, G.},
    title = "{The entropy formula for the Ricci flow and its geometric applications}",
  journal = {arxiv.org/abs/math/0211159},
   eprint = {math/0211159},
     year = 2002,
    month = nov,
   adsurl = {http://adsabs.harvard.edu/abs/2002math.....11159P}
}

@book {peter,
    AUTHOR = {Petersen, Peter},
     TITLE = {Riemannian geometry},
    SERIES = {Graduate Texts in Mathematics},
    VOLUME = {171},
   EDITION = {Third},
 PUBLISHER = {Springer, Cham},
      YEAR = {2016},
     PAGES = {xviii+499},
      ISBN = {978-3-319-26652-7; 978-3-319-26654-1},
   MRCLASS = {53-01 (53C20 53C21 53C23)},
  MRNUMBER = {3469435},
       DOI = {10.1007/978-3-319-26654-1},
       URL = {https://doi-org.proxy.library.cornell.edu/10.1007/978-3-319-26654-1},
}

@misc{songnotes,
    title={Lecture notes on the Kahler-Ricci flow},
    author={Jian Song and Ben Weinkove},
    year={2012},
    eprint={1212.3653},
    archivePrefix={arXiv},
    primaryClass={math.DG}
}

@article {tianchern,
    AUTHOR = {Tian, G.},
     TITLE = {On {C}alabi's conjecture for complex surfaces with positive
              first {C}hern class},
   JOURNAL = {Invent. Math.},
  FJOURNAL = {Inventiones Mathematicae},
    VOLUME = {101},
      YEAR = {1990},
    NUMBER = {1},
     PAGES = {101--172},
      ISSN = {0020-9910},
   MRCLASS = {32L07 (32F07 53C25 53C55)},
  MRNUMBER = {1055713},
MRREVIEWER = {M. Kalka},
       DOI = {10.1007/BF01231499},
       URL = {https://doi-org.proxy.library.cornell.edu/10.1007/BF01231499},
}

@article {tianfano,
    AUTHOR = {Tian, Gang and Zhang, Zhenlei},
     TITLE = {Regularity of {K}\"{a}hler-{R}icci flows on {F}ano manifolds},
   JOURNAL = {Acta Math.},
  FJOURNAL = {Acta Mathematica},
    VOLUME = {216},
      YEAR = {2016},
    NUMBER = {1},
     PAGES = {127--176},
      ISSN = {0001-5962},
   MRCLASS = {53C44 (32Q15 53C55)},
  MRNUMBER = {3508220},
MRREVIEWER = {Shouwen Fang},
       DOI = {10.1007/s11511-016-0137-1},
       URL = {https://doi-org.proxy.library.cornell.edu/10.1007/s11511-016-0137-1},
}

@article {topping,
    AUTHOR = {Enders, Joerg and M\"{u}ller, Reto and Topping, Peter M.},
     TITLE = {On type-{I} singularities in {R}icci flow},
   JOURNAL = {Comm. Anal. Geom.},
  FJOURNAL = {Communications in Analysis and Geometry},
    VOLUME = {19},
      YEAR = {2011},
    NUMBER = {5},
     PAGES = {905--922},
      ISSN = {1019-8385},
   MRCLASS = {53C44},
  MRNUMBER = {2886712},
MRREVIEWER = {Juan-Ru Gu},
       DOI = {10.4310/CAG.2011.v19.n5.a4},
       URL = {https://doi.org/10.4310/CAG.2011.v19.n5.a4},
}

@article {uhlenbeck,
    AUTHOR = {Uhlenbeck, Karen K.},
     TITLE = {Removable singularities in {Y}ang-{M}ills fields},
   JOURNAL = {Comm. Math. Phys.},
  FJOURNAL = {Communications in Mathematical Physics},
    VOLUME = {83},
      YEAR = {1982},
    NUMBER = {1},
     PAGES = {11--29},
      ISSN = {0010-3616},
   MRCLASS = {53C05 (58E20 81E10)},
  MRNUMBER = {648355},
MRREVIEWER = {Wolfgang L\"{u}cke},
       URL = {http://projecteuclid.org/euclid.cmp/1103920742},
}

@article {volumecold,
    AUTHOR = {Colding, Tobias H.},
     TITLE = {Ricci curvature and volume convergence},
   JOURNAL = {Ann. of Math. (2)},
  FJOURNAL = {Annals of Mathematics. Second Series},
    VOLUME = {145},
      YEAR = {1997},
    NUMBER = {3},
     PAGES = {477--501},
      ISSN = {0003-486X},
   MRCLASS = {53C21 (53C23)},
  MRNUMBER = {1454700},
MRREVIEWER = {Zhongmin Shen},
       DOI = {10.2307/2951841},
       URL = {https://doi-org.proxy.library.cornell.edu/10.2307/2951841},
}

@article {warpedrigid,
    AUTHOR = {Cheeger, Jeff and Colding, Tobias H.},
     TITLE = {Lower bounds on {R}icci curvature and the almost rigidity of
              warped products},
   JOURNAL = {Ann. of Math. (2)},
  FJOURNAL = {Annals of Mathematics. Second Series},
    VOLUME = {144},
      YEAR = {1996},
    NUMBER = {1},
     PAGES = {189--237},
      ISSN = {0003-486X},
   MRCLASS = {53C21 (53C20 53C23)},
  MRNUMBER = {1405949},
MRREVIEWER = {Joseph E. Borzellino},
       DOI = {10.2307/2118589},
       URL = {https://doi-org.proxy.library.cornell.edu/10.2307/2118589},
}

@article {zhangdegen,
    AUTHOR = {Zhang, Zhenlei},
     TITLE = {Degeneration of shrinking {R}icci solitons},
   JOURNAL = {Int. Math. Res. Not. IMRN},
  FJOURNAL = {International Mathematics Research Notices. IMRN},
      YEAR = {2010},
    NUMBER = {21},
     PAGES = {4137--4158},
      ISSN = {1073-7928},
   MRCLASS = {53C23 (53C21 53C44)},
  MRNUMBER = {2738353},
MRREVIEWER = {Manuel Fern\'{a}ndez-L\'{o}pez},
       DOI = {10.1093/imrn/rnq020},
       URL = {https://doi-org.proxy.library.cornell.edu/10.1093/imrn/rnq020},
}

@article {zvol,
    AUTHOR = {Zhang, Qi S.},
     TITLE = {Bounds on volume growth of geodesic balls under {R}icci flow},
   JOURNAL = {Math. Res. Lett.},
  FJOURNAL = {Mathematical Research Letters},
    VOLUME = {19},
      YEAR = {2012},
    NUMBER = {1},
     PAGES = {245--253},
      ISSN = {1073-2780},
   MRCLASS = {53C44},
  MRNUMBER = {2923189},
MRREVIEWER = {En-Tao Zhao},
       DOI = {10.4310/MRL.2012.v19.n1.a19},
       URL = {https://doi-org.proxy.library.cornell.edu/10.4310/MRL.2012.v19.n1.a19},
}

\end{document}